  \documentclass[11pt]{amsart}

\usepackage[colorlinks,linkcolor={blue}]{hyperref}
\usepackage{amssymb}
 \usepackage{amscd}
 \usepackage{enumerate}

\usepackage[colorinlistoftodos, textsize=footnotesize]{todonotes}
\makeatletter
\providecommand\@dotsep{5}
\makeatother
\swapnumbers

 \def\a{\alpha}

 \def\be{\beta}
 \def\hbe{{\widehat{\beta}}}
 
 \def\de{\delta}

 \def\e{\epsilon}
 \def\ve{\varepsilon}

 \def\ga{\gamma}

 \def\Ga{\Gamma}

 \def\vr{\varphi}

 \def\la{\lambda}
 
 \def\La{\Lambda}
 
 \def\Si{\Sigma}
 
 \def\om{\omega}

 \def\th{\theta}

 \def\re{{\mathbb R}}

 \def\then{\Longrightarrow}
 \def\ov{\overline}
 \def\Z{{\mathbb Z}}

 \def\co{{\mathbb C}}

 \def\hf{{\widehat{f}}}

 \def\SS{{\mathbb S}}
 \def\T{{\mathbb T}}

 \def\hU{{\widehat{U}}}

 \def\tq1{{\tilde{q}_1}}

 \def \lV{\left\Vert}
 \def \rV{\right\Vert}
 \def \ov{\overline}
 
 \def \then{\Longrightarrow}

 \def\ted{\hfill$\triangle$}

 \DeclareMathOperator*{\tsum}{{\textstyle \sum}}

  \renewcommand{\proofname}{{\bf Proof:}}

 \theoremstyle{plain}

 \newtheorem{Thm}{Theorem}[section]
 
 \newtheorem{Lemma}[Thm]{\bf Lemma}
 
 \newtheorem{Corollary}[Thm]{\bf Corollary}
 \newtheorem{Theorem}[Thm]{\bf Theorem}

 \theoremstyle{definition}

 \theoremstyle{remark}
 \newtheorem{Remark}[Thm]{\bf Remark}

 \newtheoremstyle{Cl}% name
  {5pt}%      Space above
  {3pt}%      Space below
  {\sl}%   Body font
  {}%         Indent amount (empty = no indent, \parindent = para indent)
  {\it}% Thm head font
  {:}%        Punctuation after thm head
  {.5em}%     Space after thm head: " " = normal interword space;
  {}%         Thm head spec (can be left empty, meaning `normal')

 \theoremstyle{Cl}
 \newtheorem{Claim}[equation]{Claim}

 \def\begincproof{
                  \renewcommand{\proofname}{\it Proof:}
                  \begin{proof}
                 }

 \def\endcproof{
                \renewcommand{\qedsymbol}{$\diamondsuit$}
                \end{proof} 
                \renewcommand{\qedsymbol}{\openbox}
                \renewcommand{\proofname}{\bf Proof:}
               }

 \def\iitem{\refstepcounter{equation}\item}

\makeatletter
\@namedef{subjclassname@2020}{\textup{2020} Mathematics Subject Classification}
\makeatother

 \title{No elliptic points from fixed prime ends} 

\author{Fernando Oliveira}
\address{Fernando Oliveira\newline\indent 
Universidade Federal de Minas Gerais\newline\indent
Av. Ant\^onio Carlos 6627, 
31270-901, Belo Horizonte,
MG, Brasil.}

\author{Gonzalo Contreras}  
\address{Gonzalo Contreras\newline\indent 
Centro de Investigaci\'on en Matem\'aticas\newline\indent 
A.P. 402, 36.000, Guanajuato, GTO, Mexico}

\thanks{Gonzalo Contreras was partially supported by CONACYT, Mexico, grant A1-S-10145.}

\subjclass[2020]{37E30, 37C29}

\begin{document}

\parskip +3pt

\begin{abstract}
We consider area preserving maps of surfaces and extend Mather's result
on the equality of the closure of the four branches of saddles.
He assumed elliptic fixed points to be Moser stable, while we require only that 
the derivative at this points to be a rotation by an angle different from zero.
There are many results in the literature which require the hypothesis 
that elliptic periodic points be Moser stable that now can be extended 
to the case that the derivative at these points be an irrational rotation. The key point 
is to give more information on Cartwright and Littlewood's fixed point theorem,
to show that the fixed point obtained by a fixed prime end can not be elliptic.
Hypotheses then became easier to verify: non degeneracy of fixed points and 
nonexistence of saddle connections. As an application we show that the result 
immediately implies that for the standard map family, for all values of the parameter,
except one, the principal hyperbolic fixed point has homoclinic points.
We also extend results to surfaces with boundary in order to be applicable
to return maps to surfaces of section and broken book decompositions.
\end{abstract}

\maketitle

\tableofcontents

\section{Introduction.}

  Let $S$ be a surface provided with a finite Borel measure $\mu$ 
  which is positive on non-empty open subsets. Let $f:S\to S$ be a homeomorphism.
  We say that $f$ is area preserving if $f_*\mu=\mu$.
  
   Let $p$ be a fixed point of $f$.
  We say that $p$ is of {\it saddle type} 
  if in a neighborhood $V$ of $p$ there exist continuous
  coordinates with $p$ at the origin and in which 
  $f(x,y)=(\la x,\la^{-1}y)$ with $|\la|>1$.
  The {\it stable} and {\it unstable invariant manifolds} of $p$
  are defined as:
  \begin{align*}
  W^s_p&=\{\,x\in S \,:\, \lim\nolimits_{n\to+\infty}f^n x = p\,\} \qquad \text{and}
  \\
  W^u_p&=\{\,x\in S\,:\,\lim\nolimits_{n\to-\infty}f^n x = p\,\} \quad \text{ respectively.}
  \end{align*}
  The sets $W^s_p$ and $W^u_p$ are injectively immersed connected curves.
  The components of $W^s_p-\{p\}$ and $W^u_p-\{p\}$ with respect 
  to the topology induced by a   parameterization are called {\it branches}.
  The components of $W^s_p\cap V$ and $W^u_p\cap V$ which contain $p$
  are called the {\it local invariant manifolds } of $p$ (with respect to $V$).
  The {\it local branches} of $p$ are the components of the complement of $p$
  in the local invariant manifolds. 
  
  The set $\{(x,y)\in V\,|\, x\ne 0\text{ and }y\ne 0\,\}$
  has four connected components that contain $p$ in their closures.
  We call them {\it sectors of $p$}.
  If $\Si$ is one of these sectors and $\Si'$ is a sector of $p$
   defined by means of another neighborhood $V'$ of $p$ then
   either $\Si\cap\Si'=\emptyset$
   or $\Si$ and $\Si'$ define the same germ at $p$.
   We say that the set $A$ contains a sector $\Si$ if $A$ contains a set $\Si'$
   germ equivalent to $\Si$ at $p$.
   We would like to say that a set $B$ accumulates on $p$
   with points coming through a sector.
   We say that a set $B$ {\it accumulates on a sector $\Si$ of $p$} 
   if the closure of $B\cap \Si$ contains $p$.
   These definitions do not depend on the choice of $V$
   neither on the choice of the linear map $(x,y)\mapsto (\la x,\la^{-1}y)$.
   We could have taken for the coordinates any 
   pair of expanding and contracting homeomorphisms 
   of $\re$ which fix $0$.
   We say that a branch $L$ and a sector $\Si$ are {\it adjacent}
   if a local branch of $L$ is contained in the closure of $\Si$ in $S$.
   Two branches are {\it adjacent} if they are adjacent to a single sector.
   
   By a {\it connection} we mean a branch which is contained in the 
   intersection of two invariant manifolds (possibly of two different fixed points).
   
   Let $L$ be an invariant branch and let $x\in L$.
   Denote by $D$ the closed arc from $x$ to $f(x)$ inside $L$.
   We call the set 
   $$
   \om(L):=\{y\in S\;|\;y=\lim_{n\to\infty}f^{n_i}(x_i)
   \text{ where $n_i\nearrow \infty$ and $x_i\in D$}\,\}
   $$
   the {\it limit set of $L$}. 
   The limit set $\om(L)$ is non empty, connected, compact, invariant and 
   the closure of $L$ in $S$ is the union  of $L$ and $\om(L)$.
   
   We say that a fixed point $p$ is {\it elliptic} if in a neighborhood $V$ of $p$ there exist
   continuous coordinates in which $f$ is differentiable at $p$ and $df_p$ is a rotation
   by an angle different from zero. If $p$ is of saddle type or elliptic we say that it is 
   {\it non degenerate}. We say that $p$ is {\it irrationally elliptic} if $df_p$ is a rotation
   by an angle $\theta$ where $\frac\theta{2\pi}$ is irrational.
   
   The fixed point $p$ is {\it Moser stable} if there is a 
   fundamental system of neighborhoods of $p$ made of disks whose
   frontiers are minimal invariant sets.
   
   If $p$ is a periodic point and $\tau=\inf\{\,n\ge 1\;|\;f^n(p)=p\,\}$ is its {\it period}, 
   we use $f^\tau$ to define these concepts for $p$.
   
   Now we would like to state a result of Mather \cite{Mat9}:
   
   \begin{Theorem}[Mather]\label{T1}
   \quad
   
   Let $S$ be a compact connected orientable surface provided with a finite measure $\mu$
   which is positive on open sets. Let $f:S\to S$ be an orientation preserving, area preserving  homeomorphism
   of $S$. Assume that all fixed points of $f$ are of saddle type or Moser stable and that $f$
   has no invariant connections.
   
   If $p$ is a fixed point of saddle type whose branches are invariant sets, then the four
   branches of $p$ have the same closure.
   
   \end{Theorem}
   
   We would like to remark that  the above result is false without the assumption
   that the branches of the fixed point be invariant (see the example after
    theorem 5.2 in \cite{Mat9}). The proof is based on his extension of 
    the classical theory of prime ends to open connected subsets $U$ of $S$ 
    with finitely many ideal boundary points (or topological ends). 
    This is done by adding to $U$ a circle of prime ends corresponding
    to each non trivial ideal boundary point. Let $\hU$ be the prime ends 
    compactification of $U$. Then $\hU$ is a compact surface and if $U$
    is invariant then the restriction of $f$ to $U$ extends to a homeomorphism 
    $\hf:\hU\to\hU$. If $S$ is orientable and $f$ is orientation preserving then so 
    is $\hf$.

    Cartwright and Littlewood~\cite{CL} proved that 
    if there exists a prime end $e$ fixed by $\hf$, then every point in the principal set of $e$ 
    is fixed by $f$.
    
    If $A\subset B$, we use the notation $int_B(A)$, $cl_B(A)$ and $fr_B(A)$
    for the interior, closure and the frontier of $A$ in $B$, respectively.
    The boundary of a manifold $M$ will be denoted by $\partial M$.
    By a domain we mean a connected open subset of $S$.
    
    Our result is the following:

    \begin{Theorem}\quad\label{T2}
    
    Let $S$ be a compact connected orientable surface provided with a finite
    measure $\mu$ which is positive on open sets and let $f:S\to S$ be an orientation
    preserving, area preserving homeomorphism of $S$.
    \begin{enumerate}
    \iitem\label{t21}
    Suppose that $U$ is an invariant domain of $S$ with finitely many ideal boundary points.
    Assume that all fixed points of $f$ in $fr_SU$ are non degenerate.
    If $e$ is a fixed regular prime end and $p$ is a principal point of $e$
    (which is fixed by Cartwright-Littlewood's theorem), then $p$ can not be elliptic.
    Furthermore, one of the branches of $p$ is an invariant connection contained in $fr_SU$.
    
    \iitem\label{t22} Suppose that $L$ is an invariant branch of $f$ and that all fixed points of $f$
    contained in $cl_SL$ are non degenerate. Then either $L$ is a connection or $L$
    accumulates on both adjacent branches 
    through the adjacent sectors.
    In the later alternative $L\subset\om(L)$.
    
    \iitem\label{t23} Let $p$ be a fixed point of $f$ of saddle type and let  $L_1$ and $L_2$ be adjacent
    branches of $p$ that are not connections. If $L_1$ and $L_2$ 
    are invariant and all fixed points of $f$ contained in 
    $cl_S(L_1\cup L_2)$ are non degenerate, then
    $cl_SL_1=cl_SL_2$.
    
    \iitem\label{t24} Let $p$ be a fixed point of $f$ of saddle type. Assume that all fixed points of
    $f$ contained in $cl_S(W^u_p\cup W^s_p)$ are non degenerate.
    If the branches of $p$ are invariant and none of them is a connection,
    then all the branches of $p$  have the same closure in $S$.
    
    \end{enumerate}

    \end{Theorem}

    Allowing the existence of some degenerate fixed points is useful in applications.
    Item~\eqref{t22} gives a strong dichotomy between recurrent and 
    non recurrent behaviour of $L$.

    As we shall see, each boundary component of $\hU$ is a circle periodic under $\hf$.  
    Now consider the case when $U$ is periodic of period $n$. Let $C$ be a 
    boundary component of $\hU$ and $k$ the smallest positive number such that 
    $f^{nk}(C)=C$. We define the rotation number of $U$ at $C$ as the rotation number of 
    $(\widehat{f^n})^k$ restricted to $C$. 
                We call these the 
                 {\it boundary
                   rotation numbers} of $U$.

    Mather was interested in the differentiable setting.
    
    Assume  $S$ and $\mu$ to be smooth.  
    For $1\le r\le \infty$ let $D^r_\mu(S)$ be the space of $C^r$ 
    area preserving diffeomorphisms $f$ of $S$ with the $C^r$
    topology.
    
    For $f\in D^1_\mu(S)$ consider the following properties:
    \begin{enumerate}
    \iitem\label{m1} Every periodic point of $f$ of period $\tau$ is non degenerate 
    (i.e. $1$ is not an eigenvalue of $d(f^\tau)_p$).
    \iitem\label{m2} Every elliptic periodic point is irrationally elliptic.
    \iitem\label{m3} Every elliptic periodic point of $f$ is Moser stable.
    \iitem\label{m4} $f$ has no periodic connections.
    \end{enumerate}
    It is well know that properties \eqref{m1}, \eqref{m2} 
    and \eqref{m4} are generic for $1\le r\le\infty$.
    Property~\eqref{m3} is generic for $r\ge 16$ 
    (see remark just before Theorem 6.4 of \cite{FLC}).
    
    From Theorem~\ref{T1} and standard category arguments, 
    we have that if a diffeomorphism satisfies \eqref{m1}, \eqref{m3} and~\eqref{m4},
    then the branches of every hyperbolic periodic point have the same closure. 
    This holds generically when $r\ge 16$. This was Mather's original genericity result. 
    (He also proved that generically the boundary rotation numbers 
    of every periodic open connected subset with finitely 
    many ideal boundary points are irrational).

    From Theorem~\ref{T2} we have that if a diffeomorphism satisfies 
    \eqref{m1}, \eqref{m2} and \eqref{m4}, then the branches of every hyperbolic
    periodic point have the same closure and the rotation numbers of 
    every periodic open connected subset with finitely many ideal boundary points 
    are irrational. So we have the following:

  \begin{Theorem}\label{T3}\quad
  
  Let $S$ be a compact connected orientable smooth surface provided
  with a finite smooth measure $\mu$. For $1\le r\le \infty$ let $D^r_\mu(S)$ 
  be the space of $C^r$ area preserving diffeomorphisms $f$ of $S$
  with the $C^r$ topology.
  
  Let $R$ be the set of $f\in D^r_\mu(S)$  that satisfy \eqref{m1}, \eqref{m2} and \eqref{m4}.
  Then $R$ is a residual subset of $D^r_\mu(S)$ 
  and every $f\in R$ satisfies the following:
  \begin{enumerate}%[(a)]
  \iitem\label{3a} The branches of each hyperbolic periodic point have the same closure.
  \iitem\label{3b} The boundary rotation numbers of every periodic domain with finitely many ideal boundary points are irrational.
  \end{enumerate}
  \end{Theorem}

  For fixed points we have the following:
  
  \begin{Theorem}\label{T4}
  \quad
  
  Let $S$ be a compact connected orientable smooth surface provided 
  with a finite smooth measure.
  
  Let $A$ be the set of $f\in D^r_\mu(S)$ such that every fixed point of $f$
  is non degenerate and $f$ has no invariant connections.
  The set $A$ is $C^1$ open and $C^r$ dense for every $r$.
  
  If $f\in A$ then the branches of each hyperbolic fixed point have the same closure.
  \end{Theorem}
  
  In section~\ref{S2} we present the facts about prime ends we need. In section~\ref{S3}
  we prove Theorem~\ref{T2}. In section~\ref{S4} we extend results to surfaces with boundary
  and present some results about the existence of homoclinic points. In section~\ref{s5} 
   we give
  applications to the standard map family. 
  
  In section~\ref{S4} we prove the results about homoclinic points for 
  partially defined area preserving
  maps of surfaces with boundary 
  in order to be applicable to Poincar\'e maps
  of surfaces of section of Reeb flows and to 
  holonomy maps of broken book decompositions.
  In \cite{CO3} we use these results to prove that
  for a Kupka-Smale riemannian metric on a closed surface
  every hyperbolic geodesic has homoclinics in all its branches.
  
  \section{Prime ends.}\label{S2}

 In sections~\ref{S2} and ~\ref{S3}, $S$ will be a compact connected orientable surface without boundary.
 Now we are going to describe the theory of prime ends of Mather~\cite{Mat9} and~\cite{Mat12}.

 \subsection{The ideal boundary.}\quad
 
 We will be interested in the connected components of the complement of the closure of invariant manifolds and branches. If $K$ is a compact connected subset of $S$, a {\it residual domain of $K$}
 is a component of $S- K$.
 
 Let $U$ be a connected surface without boundary. 
 We are going to describe a 
 \linebreak
 compactification of $U$ 
 by the addition of its  topological ends or boundary components.
 See~\cite{Massey}, \cite{Rich}.
 
 A {\it boundary component of $U$} is a 
 decreasing sequence $P_1\supset P_2\supset\cdots$
 of open connected subsets of $U$ such that:
 \begin{enumerate}
 \iitem\label{b1}
 Every $P_n$ is not relatively compact.
 \iitem\label{b2}
 Every $fr_U(P_n)$ is compact.
 \iitem\label{b3}
 If $K$ is a compact subset of $U$, 
 then there is $n_0$ such that $K\cap P_n=\emptyset$
 for $n \ge n_0$.
 \end{enumerate}

Two boundary components $P_1\supset P_2\supset\cdots$ and 
$P_1'\supset P_2'\supset \cdots$ are equivalent if 
for every $n$ there is $m$ such that $P_m\subset P'_n$ and vice versa.
An {\it ideal boundary point} is an equivalence class of boundary components. 
The set of ideal boundary points $b_IU$ is called 
the {\it ideal boundary} of $U$ and the disjoint union
$U^*=U\cup b_IU$ is called the {\it ideal completion} of $U$.
We have that $b_IU=\emptyset$ if and only if $U$ is compact.
Let $A$ be an open subset of $U$ with $fr_U(A)$ compact and 
let $A'$ be the set of ideal boundary points whose representing
boundary components $(P_n)$ satisfy $P_n\subset A$ for $n$
larger than some $n_0$. 
The collection of sets $A\cup A'$ where $A$ is an open 
subset of $U$ with $fr_U(A)$ compact forms a basis for a topology of $U^*$.

The space $U^*$ is a compactification of $U$ characterized by the facts that 
$U^*$ is locally connected and $b_IU$ is totally disconnected and nonseparating
on $U$ (i.e. for any open connected subset $V$ of $U^*$, $V-b_IU$ is connected),
(cf. \cite[\S 36--37, ch.1]{AhSa}).

If $U$ is a domain of $S$ and $b_IU$ is finite, 
then $U^*$ is a connected compact orientable surface without boundary
(Proposition 2.1 in~\cite{Mat9} or Prop.~3.12 in~\cite{OC2}).
This happens for residual domains (Lemma~2.3 of \cite{Mat9}).
 
 For $b\in b_IU$ let 
 $$
 Z(b):=\textstyle\bigcap\nolimits_i cl_S(N_i\cap U),
  $$  
  where $(N_i)$ is a fundamental system of neighborhoods of $b$ in $U^*$.
  Then $Z(b)$ is just the set of limit points in $S$ of sequences in $U$ that 
  converge to $b$ in $U^*$.
  The set $Z(b)$ is non-empty, connected, compact and $fr_S(U)=\bigcup_{b\in b_IU}Z(b)$.
  If $Z(b)$ contains more than one point we say that $b$ is a {\it regular} ideal boundary point of $U$.

  \subsection{The prime ends compactification.}\label{ss22}\quad
  
  Let $U$ be a domain of $S$ such that $b_IU$ is finite.
  Now we describe another compactification of $U$ 
  by replacing each regular ideal boundary point by a circle.
  
  A {\it chain} is a sequence $V_1\supset V_2\supset \cdots$ 
  of open connected subsets of $U$ such that 
  \begin{enumerate}
  \iitem\label{chn1} $fr_U(V_i)$ is non-empty and connected for every $i\ge 1$, and 
  \iitem\label{chn2} $cl_S(fr_U(V_i))\cap cl_S(fr_U(V_j))=\emptyset$ for $i\ne j$.
  \end{enumerate}
  A chain $(W_j)$ {\it divides } $(V_i)$ if for every $i$ there exists $j$ 
  such that $W_j\subset V_i$.
  Two chains are equivalent of each one divides the other.
  A chain is {\it prime} if any chain which divides it is equivalent to it.
  A sufficient condition for a chain $(V_i)$ to be prime is that there exists
  $p\in S$ such that $fr_U(V_i)\to p$, that is, every neighborhood
  of $p$ contains all but finitely many of the sets $fr_U(V_i)$.
  
  A {\it prime point} is an equivalence class of prime chains.
  Let $x\in U$ and consider a family of closed disks $D_1\supset D_2\supset \cdots$ 
  in $U$, such that $D_{i+1}\subset int_U(D_i)$ and $\cap_i D_i=\{x\}$.
  Then $(int_U(D_i))$ is a prime chain that defines a prime point denoted by $\om(x)$.
  A {\it prime end} is a prime point which is not of the form $\om(x)$ for any $x\in U$.
   
  The set of prime points of $U$ is denoted by $\hU$. We consider a topology on $\hU$
  defined as follows. Let $V$ be an open subset of $U$ 
  and denote by $V'$ the set of prime points of $U$ 
  whose representing chains are eventually contained in $V$.
  The collection of sets $V'\subset \hU$ such that $V$ is an open subset of $U$
  is a basis for the topology on $\hU$.
  The function $\om: U\to \hU$ is a homeomorphism from $U$ onto an open subset of $\hU$,
  and we identify $U$ with $\om(U)$.
  
  If $e\in\hU$, let 
  $$
  \a(e):=\textstyle\bigcap_icl_{U^*}(V_i)\,,
  $$
  where $(V_i)$ is a chain representing $e$.
  Then $\a(e)$ consists of one point and $\a:\hU\to U^*$ 
  is a continuous function whose restriction to $U$
  is the inclusion ($\a\circ\om$ is the identity on $U$).
  A prime end $e\in \hU-U$ is said regular if 
  $Z(\a(e))$ contains more than one point.
  
  The set $\hU$ is a compact connected surface with boundary.
  This is Theorem~10 of \cite{Mat12}.
  $\hU-U$ is the set of prime ends of $U$, the boundary $\partial\hU$
  is the set of regular prime ends of $U$
  and the boundary components of $\hU$ 
  are the sets $\a^{-1}(b)$ where $b$ varies on the set 
  of regular ideal boundary points of $U$.
  The set $\a^{-1}(b)$ is homeomorphic to a circle called
  the {\it Caratheodory circle associated with $b$}.
  We denote it by $C(b)$.
  
  Let $e\in\hU$ and let $(V_i)$ be a chain representing $e$.
  The set 
  $$
  Y(e)=\textstyle\bigcap_i cl_S(V_i)
  $$
  is called the {\it impression of $e$}.
  The definition does not depend on the representing chain, 
  and $Y(e)$ is a compact, connected, non-empty subset of $S$.

  We say that $p\in S$ is a {\it principal point} of $e$ if there is a chain $(V_i)$
  which represents $e$ and for which $fr_U(V_i)\to p$ 
  (given a neighborhood $V$ of $p$ there exists $n_0$ such that
  $fr_U(V_i)\subset V$ for all  $n\ge n_0$).
  The set of principal points of $e$ is called the 
  {\it principal set} of $e$ and is denoted by $X(e)$.
  It is a non-empty, compact, connected subset of $S$, and 
  $$
  X(e)\subset Y(e)\subset Z(\a(e)).
  $$
  
  If $X(e)=\{p\}$ for some $p\in fr_S(U)$, we say that $e$ is an {\it accessible prime end}.
  The point $e\in\hU$ is an accessible prime end if and only if
  there exists a path $\be:]0,1]\to U$ such that $\lim_{t\to 0}\be(t)=p$ in $S$
  and $\lim_{t\to 0}\be(t)=e$ in $\hU$.
  
  \begin{Lemma}\label{L1}\quad
  
  Let $U$ be an open connected set.
  Let $A_1$ and $A_2$ be open connected subsets of $U$, both non-empty and 
  different from $U$.
  Suppose that 
  \begin{enumerate}
  \iitem\label{l11}
  $A_1\cap A_2\ne\emptyset$.
  \iitem\label{l12}
  $A_1\cap (U-A_2)\ne \emptyset$.
  \iitem\label{l13}
  $A_2\cap(U-A_1)\ne\emptyset$.
  \iitem\label{l14}
  $fr_U(A_1)$ and $fr_U(A_2)$ are connected and disjoint.
  \end{enumerate}
  Then we have the following facts:
  \begin{enumerate}%[\quad\, (a)]
  \iitem\label{l1a}
  $fr_U(A_1)\subset A_2$ and $fr_U(A_2)\subset A_1$.
  \iitem\label{l1b}
  $U=A_1\cup A_2$.
  \end{enumerate}

  \end{Lemma}
  
  \begin{proof}
  
  Since $fr_U(A_1)$ is connected and does not intersect $fr_U(A_2)$, we must have either $fr_U(A_1)\subset A_2$ or $fr_U(A_1)\subset U- A_2$.
  But $A_2$ is a connected set that instersects both 
  $A_1$ and $U-A_1$, and therefore we have $A_2\cap fr_U(A_1)\ne\emptyset$.
  This implies that $fr_U(A_1)\subset A_2$.
  Similarly, $fr_U(A_2)\subset A_1$, which proves part \eqref{l1a}.

  In orther to prove item~\eqref{l1b}, 
  we are going to show that $A_1\cup A_2$ is open and closed in $U$.
  It is obviously open. Let $(a_n)$ be a sequence in $A_1\cap A_2$ converging to a point $a$.
  We may assume that $a_n\in A_1$ for infinitely many values of $n$, and therefore
  $a\in cl_U(A_1)$. If $a\in A_1$ we are done. If not, we have that 
  $a\in fr_U(A_1)$, and from item~\eqref{l1a}, $a\in A_2$.
  
  \end{proof}

\subsection{An useful lemma.}\quad

We would like to show that accessible prime ends can be represented
by chains whose frontiers are contained in arcs of circles.

For this we are going to consider $\re^2$ equipped with a norm
$\lV \;\,\rV$. We denote by $B_r$, $B^\circ_r$ and $C_r$ the closed ball, 
the open ball and the circle with center at $(0,0)$ and radius $r$, respectively.
By a closed disk $D$ we mean a set homeomorphic to a closed disk in $\re^2$.
We write $\partial D$ for its boundary and $D^\circ$ for its interior $D-\partial D$.

Let $e$ be an accessible prime end and let  $p$ be its principal point. 
In a neighborhood of $p$ we consider continuous coordinates with $p$ at the origin.

  \begin{Lemma}\label{L2}\quad
  
  Let $U$ be a domain of $S$ such that $b_IU$ is finite.
   Let $e$ be an accessible prime end of $U$ and $p$ its principal point.
  Then we have the following:
  \begin{enumerate}   
    \iitem\label{l2a}
  There exists $\de>0$ such that for any decreasing sequence $(r_n)_{n\ge 1}$
  contained in $(0,\de)$ with $\lim_{n\to \infty}r_n=0$ there exists a chain 
  $(V_n)$ representing $e$
  such that $fr_UV_n\subset C_{r_n}$.
  
  \iitem\label{l2b} 
  For every $n\ge 1$ if $\rho\in]r_{n+1},r_n[$
   then there exists an open subset $W$ of $U$ such that $fr_U W$ 
   is connected, contained in $C_\rho$ and $(W_i)$ defined by
   $W_i=V_i$ if $i\ne n+1$ and $W_{n+1}=W$ is also a chain representing $e$.
   \end{enumerate}
   \end{Lemma}
   
   \begin{proof}\quad
   
   Let $\be:]0,1]\to U$ be a path such that $\lim_{t\to 0}\be(t)=p$ in $S$ and
   $\lim_{t\to 0}\be(t)=e$ in $\hU$.
   
   If $b=\a(e)$ then $\lim_{t\to 0}\be(t)=b$ in $U^*$.
   In fact, $\cap_{\e>0\,}cl_{U^*}\be(]0,\e[)$ 
   is a non-empty, compact,  connected subset of $U^*$
   which is disjoint from $U$.
   Since $U^*-U$ is finite, $\cap_{\e>0}\,cl_{U^*}\be(]0,\e[)$
   is a single point equal to $\lim_{t\to 0}\be(t)$ in $U^*$.
   By the definition of $\a$ this point must be $\a(e)$.
   
   Let $D$ be a closed disk in $U^*$ such that $b\in D^\circ$,
   $D\cap (U^*-U)=\{ b\}$ and 
   \linebreak
   $\be(]0,1])\cap \partial D\ne\emptyset$.
    The subset $\partial D\subset U$ is compact. 
   Therefore there exists $\de>0$
   such that $B_r\cap \partial D=\emptyset$ if $r<\de$. 
   Let $t_0=\inf\{\,t\in ]0,1]\;|\;\be(t)\in\partial D\,\}$.
   Then $\be(t_0)\in\partial D$ and $\be(]0,t_0[)\subset D^\circ$.

   \begin{Claim}\label{c24}
   For any $r\in]0,\de[$, if $\be(]0,c[)\subset B^\circ_r$ 
   and $\be(d)\notin B_r$ with $d\le t_0$, 
   then there exists a connected component $\xi$ of $C_r\cap U$ such that
   $U-\xi$ is the disjoint union of two open connected subsets of $U$ 
   with $\be(]0,c[)$ contained in one 
   and $\{\be(d),\be(t_0)\}$ contained in the other.
    \end{Claim}

   \color{black}
   \noindent
   {\it Proof of the claim}.

       First note that the components of $C_r\cap U$ form an open
     cover of $\be(]0,d])\cap C_r$ which is compact. Therefore $\be(]0,d])$
      intersects only finitely many components of $C_r\cap U$, 
      say $\la_1,\ldots,\la_n$.

      Since $\lim_{t\to 0}\be(t)=b\in D^\circ$ and $\be(]0,t_0[)\cap \partial D=\emptyset$,
      we have that $\be(]0,t_0])\subset D$.
      Then $\be(]0,d])\subset D-\{b\}$,
       and hence
      $\la_i\cap(D-\{b\})\ne \emptyset$. On the other hand, since
      $\la_i\subset C_{r_i}$, we have that 
      $\la_i\cap\partial D=\emptyset$.
      Therefore $\la_i\subset D-\{ b\}$.
      We have that $\la_i$ is an open arc in $D\subset U^*$ with end points at $b$.
      This implies that $U-\la_i$ has two components.
      
           We are going to assume that  $\be(]0,c[)$ and $\be(d)$ belong
     to the same component of $U-\la_i$  for $1\le i\le n$ and 
     obtain a contradiction.

      Let $\la_0$ be the union of the components of $C_r\cap U$ different from 
      $\la_1,\ldots,\la_n$.
      
      For $0\le i\le n-1$ assume that there is a path 
      $\a:[a,d]\to U-(\la_0\cup \cdots \cup \la_i)$
      from $\be(a)$ to $\be(d)$ where $a$ is any number in $]0,c[$.
      We are going to show that there is
       another path $\ga:[a,d]\to U-(\la_0\cup\cdots\cup \la_i\cup \la_{i+1})$
      from $\be(a)$ to $\be(d)$.
      
      If $\a$ does not intersect $\la_{i+1}$ then we just take $\ga=\a$. 
      When $\a\cap \la_{i+1}\ne \emptyset$ let 
      $u=\inf\{ t\,|\,\a(t)\in\la_{i+1}\}$ and
      $v=\sup\{\,t\,|\,\a(t)\in \la_{i+1}\}$.
      Since $\be(a)$ and $\be(d)$ are in the same component 
      of $U-\la_{i+1}$,
      if $u'<u$ and $v'>v$ then 
      $\a([a,u'])$ and $\a([v',d])$
      are contained in the same component of 
      $U-\la_{i+1}$. 
      If $u'$ and $v'$ are close enough to $u$ and $v$ then the arc $\eta$
      of the circle $C_{r'}$ with end points at $\a(u')$ and $\a(v')$
      does not intersect $\la_0\cup\cdots\cup \la_i\cup \la_{i+1}$.
      If we define $\ga$ by just replacing the part of $\a$ 
      between $\a(u')$ and $\a(v')$ by $\eta$,
      we obtain the desired path
      $\ga:[a,d]\to U-(\la_0\cup \cdots \cup \la_i\cup \la_{i+1})$ 
      from $\be(a)$ to $\be(d)$.
      
      We have that  $\be(]0,d])\cap \la_0=\emptyset$.
      Therefore there exists a path 
      $\ga_1:[a,d]\to U -(\la_0\cup \la_1)$
      from $\be(a)$ to $\be(d)$.
      By induction there exists a path 
      $\ga_n:[a,d]\to U -(\la_0\cup \cdots \cup \la_n)$
      from $\be(a)$ to $\be(d)$. 
      This is a contradiction, since any path in $U$ from
      $\be(a)$ to $\be(d)$ must intersect $C_r\cap U$.
      Therefore there is some $k$ such that $\be(]0,c[)$ 
      and $\be(d)$ belong to different components of $U-\la_k$.
      
      Each $\la_i$ is an open arc in $D^\circ\subset U^*$ with end points at $b$.
      The curve $\la_i^*=\la_i\cup\{b\}$ separates $U$ into two 
      components $V_i$ and $W_i$, where $W_i$ contains $\partial D$
      and $fr_{U^*} V_i=\la_i^*$. Since $\be(d)\notin B_r$, there is a path in 
      $D-B_r$ joining $\be(d)$ to a point in $\partial D$. In particular it does not 
      intersect $\la_i$. This implies that $\be(d)\in W_i$. Also $\be(t_0)\in\partial D\subset W_i$.
      
      For $i=k$, since $\be(d)\in W_k$, we have that  $\be(]0,c[)\subset V_k$, 
      and $\{\be(d),\,\be(t_0)\}\subset W_k$.

      \ted

      We construct the chain in the following way.
     For $i\ge 1$ let 
   $$
   t_i=\inf\{\,t\in]0,1]\,|\,\be(t)\in C_{r_i}\}.
   $$
   Then 
   \begin{equation}\label{e1l2}
   \be(t_i)\in C_{r_i}\quad\text{and} \quad 
   \be(]0,t_i[)\subset B^\circ_{r_i}.
   \end{equation}
   The set $V_1$ is defined in the following way. 
   By hypothesis $r_1<\de$, then  $B_{r_1}\cap \partial D=\emptyset$.
   We have that 
   $\be(t_0)\notin B_{r_1}$ and $\be(]0,t_1[)\subset B^\circ_{r_1}$.
   Using Claim~\ref{c24}, let $\xi_1$ be a component of $C_{r_1}\cap U$ for which 
   $U-\xi_1$ is the disjoint union $V_1\cup W_1$ of two open
   connected subsets of $U$ with $\be(]0,t_1[)\subset V_1$ and $\be(t_0)\in W_1$.
   Also $\be(t_0)\notin cl_UV_1$
    because $\be(t_0) \notin B_{r_1}\supset C_{r_1}\supset fr_U V_1$ 
    and $\be(t_0)\notin V_1$.
    
    Suppose now that we have defined $V_1\supset V_2\supset \cdots \supset V_n$
    with the following properties for $1\le i\le n$:
    \begin{enumerate}   
    \iitem\label{ul1} $\be(]0,t_i[)\subset V_i$.
    \iitem\label{ul2} $fr_UV_i=\xi_i$ is a connected component of $C_{r_i}\cap U$.
    \iitem\label{ul3} $\be(t_{i-1})\notin cl_U V_i$.
    \iitem\label{ul4} $\be(t_0)\notin V_i$.
    \end{enumerate}
    The previous paragraph shows that \eqref{ul1}, \eqref{ul2}, \eqref{ul3}, \eqref{ul4} hold for $i=1$.
    
    By \eqref{e1l2} we can apply the claim to $\be(]0,t_{n+1}[)$ and $\be(t_n)$.
    Let $\xi_{n+1}$ be a component of $C_{r_{n+1}}\cap U$
     for which $U-\xi_{n+1}$ is the disjoint union $V_{n+1}\cup W_{n+1}$
     of two open connected subsets of $U$ with $\be(]0,t_{n+1}[)\subset V_{n+1}$ 
     and $\{\be(t_n),\,\be(t_0)\}\subset W_{n+1}$.
      Observe that the choice of $V_{n+1}$ and $W_{n+1}$
     implies that $fr_U V_{n+1}=fr_U W_{n+1}=\xi_{n+1}$.
     Then $\be(t_n)\notin V_{n+1}$ and $\be(t_n)\notin fr_UV_{n+1}$ because
     $fr_UV_{n+1}=\xi_{n+1}\subset C_{r_{n+1}}$ and $\be(t_n)\in C_{r_n}$.
     Therefore $\be(t_n)\notin cl_U V_{n+1}$ and $V_{n+1}$
     satisfies \eqref{ul1},  \eqref{ul2},  \eqref{ul3}, \eqref{ul4} above.

     We  show now that $V_{n+1}\subset V_n$.
     We have that $\be(]0,t_{n+1}[)\subset V_n\cap V_{n+1}$. 
     Therefore $V_n$ and $V_{n+1}$ satisfy hypothesis \eqref{l11}
     of lemma~\ref{L1}.
     Also $\xi_n\cap \xi_{n+1}=\emptyset$ and then hypothesis \eqref{l14} also holds.
     Since $\be(t_n)\notin cl_U V_{n+1}$ we have that $\be(]t_n-\de,t_n[)\cap V_{n+1}=\emptyset$
     for $\de$ small enough. Therefore $\be(]t_n-\de,t_n[)\subset V_n-V_{n+1}$
     and \eqref{l12} holds.
      By \eqref{ul4} we have that $\be(t_0)\notin V_n\cup V_{n+1}$
      and then conclusion
     \eqref{l1b} of Lemma~\ref{L1} does not hold.
      Therefore  hypothesis \eqref{l13} does not hold and then $V_{n+1}\subset V_n$.
      
     Thus the chain $(V_n)$ is well defined. Since $fr_U V_n\subset C_{r_n}$, we have that 
     $fr_U V_n\to p$ and $(V_n)$ is prime. 
     Let $e'$ be the prime end it defines.
     Since $\be(]0,t_n[)\subset V_n$ for every $n$
     we have that $\lim_{t\to 0} \be(t)=e'$ in $\hU$.
     
     Hence $e=(V_n)$. This  proves part \eqref{l2a} of the lemma.
  
       As for part \eqref{l2b}, if we perform      
       the previous construction with 
       the new sequence $s_i=r_i$ if 
       $i\ne n+1$ and $s_{n+1}=\rho$,
       then $(t_i)$, $(\xi_i)$ and $(V_i)$
       remain the same except for $i=n+1$,
       when we obtain an open subset $W$ of $U$
       such that $fr_UW$ is an arc of 
       $C_\rho\cap U$ and $V_{n+2}\subset W\subset V_n$.
      
   \end{proof}

   \subsection{Cartwright-Littlewood's Theorem.}\quad
   
   Let $f:S\to S$ be an area preserving orientation 
   preserving homeomorphism and $U$
   an invariant domain with $b_IU$ finite.
   If we denote by $f_U$ the restriction of $f$ to $U$,
   then $f_U$ has a unique extension $f^*:U^*\to U^*$, 
   which is also an orientation preserving homeomorphism.
   Points in $b_IU$ are periodic. 
   Since $f_U$ maps irreducible chains to irreducible chains,
   $f_U$ also extends to an orientation preserving homeomorphism 
   $\hf:\hU\to\hU$
   and Caratheodory circles are permuted in the same way as their associated 
   regular ideal boundary points.
   
   Suppose now that $U$ is periodic of period $n$. 
   Let $C$ be a Caratheodory circle and let  $k$ be the 
   the smallest positive number such that $f^{nk}C=C$.
   Define the rotation number of $U$ at $C$ as the 
   rotation number of $(\widehat{f^n})^k$ restricted to $C$.
   
   \begin{Lemma}\label{L3}\quad
   
   Let $e$ be a fixed prime end of $U$,
   $p$ a principal point of $e$ 
   and $(V_i)$ a chain defining $e$ such that $fr_U(V_i)\to p$.
   Let $i_0=\inf\{\,i\ge 1\,|\, fV_i\subset V_1\,\}$. Then  
   \begin{equation*}\label{l31}
   \forall i\ge i_0, 
   \qquad fr_U(V_i)\cap fr_U(fV_i)\ne\emptyset.
   \end{equation*} 
   
   \end{Lemma}

   \begin{proof}\quad
   
   We have that 
   $(fV_i)$ is a chain that represents $\hf(e)$.
   Since $e$ is a fixed point of $\hf$, we have that 
   $(fV_i)$ and $(V_i)$ are equivalent.
   
   Let $i_0$ be such that $fV_i\subset V_1$ for $i\ge i_0$. 
   We are going to assume that $fr_U(V_i)\cap fr_U(fV_i)=\emptyset$
   for some $i\ge i_0$ and apply Lemma~\ref{L1} to $V_i$ 
   and $fV_i$ to obtain a contradiction.
   
   So we are assuming that hypothesis~\eqref{l14} of Lemma~\ref{L1} holds.
   The equivalence of $(V_i)$ and $(fV_i)$ implies that 
   $V_i\cap fV_i\ne \emptyset$ for every $i\ge 1$ and hypothesis \eqref{l11}
   also holds.
   We have that $V_i\subset V_1$ for $i\ge 1$ and therefore
   $V_i\cap fV_i\subset V_1$ for $i\ge i_0$.
   Since $V_1$ is strictly contained in $U$, we have that conclusion~\eqref{l1b}
   does not hold.
   
 This implies  that hypotheses \eqref{l12} and~\eqref{l13} can not both hold.
 So we have that $V_i\subset fV_i$ or $fV_i\subset V_i$. 
 We are going to assume that $fV_i\subset V_i$ and show that 
 $\mu(fV_i)<\mu(V_i)$. The other case is analogous.
 
 Since $fr_U(fV_i)$ is connected and 
 \begin{equation}\label{l32}
 fr_U(fV_i)\cap fr_U V_i=\emptyset,
 \end{equation}
  we have that 
 $fr_U(fV_i)\subset V_i$ or 
 $fr_U(fV_i)\subset U-V_i$. 
 We claim that the case
 \linebreak
  $fr_U(fV_i)\subset U-V_i$ can not happen.
 If so, by equality   \eqref{l32}  we would have  that 
 \linebreak
 $fr_U(fV_i)\subset U-cl_U V_i$.
 If $x\in fr_U(fV_i)$ then every neighborhood $W$
 of $x$ would 
 \linebreak
 contain points of $fV_i\subset V_i$. 
 On the other hand we could choose $W$ disjoint
 from $cl_U V_i$, a contradiction.
 
 Therefore $fr_U(fV_i)\subset V_i$, implying that 
 $cl_U(fV_i)\subset V_i$. 
 Since $fr_U (V_i)\to p$, $V_i\ne U$.
 If we had $cl_U(fV_i)=V_i$ then $V_i$ 
 would be open and closed at the same time, 
 contradicting the connectivity of $U$.
 Hence $V_i-cl_U(fV_i)\ne \emptyset$.
 Since $\mu$ is positive on open sets,
 we have that $\mu(cl_U(fV_i))<\mu(V_i)$.
 A contradiction.

   \end{proof}

 Next we present Cartwright and Littlewood's fixed point theorem.
 
 \begin{Theorem}[Cartwright and Littlewood~\cite{CL}]\label{T5}\quad
 
 Let $f:S\to S$ be an area preserving orientation 
   preserving homeomorphism and $U$
   an invariant domain with $b_IU$ finite.
   Let $e$ be a fixed prime end of $U$ and $p$ a principal point of $e$.
 Then $f(p)=p$.
 \end{Theorem}
 
 \begin{proof}
 Let $(V_i)$ be a chain defining $e$ such that $fr_U (V_i)\to p$.
 From Lemma~\ref{L3} there exists $i_0$ such that 
  for  $i\ge i_0$  there exists a point $x_i\in fr_U(V_i)$
  such that $f(x_i)$ 
 also belongs to $fr_U(V_i)$.
 Since $fr_U(V_i)\to p$, we have that $x_i\to p$ and $f(x_i)\to p$,
 implying that $f(p)=p$.
 \end{proof}

\section{No elliptic points from periodic prime ends.}
\label{S3}

Now we start the proof of item~\eqref{t21}   of Theorem~\ref{T2}.
We have that  $b=\a(e)\in b_IU$ is a regular ideal boundary point,
 $e$ is a prime end that belongs to the Caratheodory circle 
 associated with $b$ and $\hf(e)=e$. 
  
 Let $p$ be a principal point of $e$.
 By Theorem~\ref{T5}, $f(p)=p$.
 Our hypothesis that fixed points of $f$ contained in $fr_SU$
 are non degenerate implies that they are isolated.
 Therefore the principal set  of $e$ reduces to $\{p\}$ and hence 
 \begin{equation}\label{accessible}
 \text{$e$
 is accessible.}
 \end{equation}
 
 \subsection{The fixed point $p$ can not be elliptic.}\quad
 
 We are going to assume that $p$ is elliptic and obtain a contradiction.
 
 Let $V$ be a neighborhood of $p$ where there exist continuous coordinates 
 with $p$ at the origin and in which $f$ is differentiable at $p$ and $df_p$ 
 is a rotation by an angle $\a$ different from zero.
 
 We  write the coordinates in the plane as complex numbers.
 By hypothesis we have that $f'(0)=e^{i\a}$ with $0<\a<2\pi$.
 
 Since $Z(b)$ is connected and contains more than one point,
 there exists a closed ball $B$ centered at $0$ such that $Z(b)$ is not contained in $B$.
 
 We are going to construct a closed curve $\xi$ satisfying three conditions:
 \begin{itemize}
 \iitem $\xi\subset U$,
 \iitem $\xi\subset B^\circ$,
 \iitem If $B'$ is the component of $S-\xi$ which contains $p$, then 
 $B'$ is homeomorphic to an open disk whose closure is contained
 in $B^\circ$ and $fr_S B'\subset \xi$.
 \end{itemize}
 
 Since $p\in X(e)$ and $X(e)\subset Z(b)$, we have that 
 $Z(b)\cap B'\ne\emptyset$. On the other hand, since
 $B'\subset B$ and $Z(b)$ is not contained in $B$, we have that 
 $Z(b)\cap (S-B')\ne \emptyset$.
 Being connected $Z(b)$ must intersect $fr_SB'$,
 a contradiction since $Z(b)\subset fr_S(U)$
 and $fr_S B'\subset \xi\subset U$.
 
 Let $B^\circ_\de$ be the open ball of radius $\de$ centered at the origin
 with $\de$ small enough so that $f(B^\circ_\de)\subset V$.
 For points $z$ close to $0$ we will need to estimate the distance
 between iterates of $z$ under $f$ and $f'(0)$.
 
 \begin{Lemma}\label{L4}\quad
 
 Let $w$ and $w'\in\co$ and $\e<1$ be such that 
 $|w'-w|<\e\,|w|$. 
 If $w=r e^{i\eta}$ and $r'=|w'|$ 
 then there exists a unique real number $\eta'$ such that 
 \begin{enumerate}%[\quad(a)\;\;]
 \iitem\label{l41} $w'=r' e^{i\eta'}$,
 \iitem\label{l42} $|r'-r|<\e\, r$,
 \iitem\label{l43} $|\eta'-\eta|<\tfrac\pi 2 \,\e$.
 \end{enumerate}

 \end{Lemma}

\begin{proof}\quad

Let $B_r$ be the closed ball of radius $\e r$ with center at $w$.
Then $w'$ belongs to the interior of $B_r$ and $0\notin B_r$.

Let $\nu$ be the angle between the ray starting at $0$ 
through $w$ and one of the two rays starting at $0$ and tangent to $B_r$.
Since $0\notin B_r$, we have that $\nu<\frac \pi 2$.
Then there is a unique real number $\eta'$ such that $w'=r' e^{i\eta'}$ and 
$|\eta'-\eta|<\nu$.
For $0<\nu<\frac \pi 2$ we have that $\frac 2\pi\nu<\sin\nu$. But $\sin\nu=\e$. 
Therefore $|\eta'-\eta|<\frac\pi 2\e$ which proves 
\eqref{l41} and \eqref{l43}. Item \eqref{l42} is just the triangle inequality.

\end{proof}

For $\e\in]0,1[$ let $\de>0$ be such that if $|z|<\de$ then 
$|f(z)-f'(0)z|<\e\,|z|$.

\begin{Lemma}\label{L5}\quad

Let $z=r e^{i\th}$ with $r<\de$ and let  $r'=|f(z)|$. 
Then there exists a unique real number $\th'$ such that 
\begin{enumerate}%[\quad (a)\;\;]
\iitem\label{l51}
$f(z)=r' e^{i\th'}$,
\iitem\label{l52}
$|r'-r|<\e\, r$,
\iitem\label{l53}
$|\th'-(\th+\a)|<\frac\pi 2 \e$.
\end{enumerate}

\end{Lemma}

\begin{proof}\quad

Let $w'=f(z)$ and $w=f'(0)z$. Then $w=r e^{i(\th+\a)}$ and $|w'-w|<\e\, |w|$.
Let $\eta=\th+\a$. Lemma~\ref{L4} provides a unique number $\th'=\eta'$ for 
which~\eqref{l51}, \eqref{l52} and \eqref{l53} are satisfied.

\end{proof}

By Lemma~\ref{L5}, we have that
 $F:\re\times]0,\de[\to\re\times]0,+\infty[$ defined by $F(\th,r)=(\th',r')$
 is a lifting of $f:B_\de-\{0\}\to\co$ with the property that $|r-r'|<\e r$
 and $|\th'-(\th+\a)|<\frac\pi 2\e$ for every $(\th,r)\in\re\times]0,\de[$. 
 We are going to use $(\th_j,r_j)=F^j(\th_0,r_0)$ to denote iterates
 of a point $(\th_0,r_0)$.
 
 Now we start the construction of $\xi$.
 
 Let $B$ be a closed ball centered at $0$ such that $Z(b)$ is not contained in $B$.
 
 Let $n$ be such that $n\a>2\pi$ and $n(2\pi-\a)>2\pi$.
 
 Let $\e\in]0,1[$ be such that 
 \begin{gather}
 n \tfrac\pi 2 \e < n\a - 2\pi,
 \label{e3}
 \\
 n\tfrac \pi 2 \e < n(2\pi-\a) -2\pi,
 \label{e4}
 \\
 \a+\tfrac \pi 2 \e < 2\pi \qquad\text{and}\qquad \a-\tfrac\pi 2\e >0.
 \label{e5}
 \end{gather}
 
 Let $\de>0$ be such that $B_\de\subset B^\circ$,
 $f(B^\circ_\de)\subset V$ and if $|z|<\de$ then 
 $|f(z)-f'(0)z|<\e|z|$.
 
 If $r_0$ satisfies $r_0(1+\e)^n<\de$, then from \eqref{l52} of Lemma~\ref{L5} 
 we have that 
 \begin{equation}\label{e55}
 r_j<r_0\,(1+\e)^j<\de
 \end{equation}
  and 
 $F^j:\re\times]0,\de[\to\re\times]0,+\infty[$ is well defined for 
 $1\le j\le n$.
 From \eqref{l53} of Lemma~\ref{L5} we have that 
 \begin{equation}\label{e6}
 |\th_j-(\th_0+j\a)|<j\tfrac\pi 2\e
 \qquad \text{for}\quad 1\le j\le n.
 \end{equation}
 
 By Lemma~\ref{L2} there exists a chain $(V_i)$ representing $e$
 such that $fr_U V_i\to p$ and the sets $fr_U V_i$ 
 are all contained in circles with center at $0$.
 For $i$ large enough let $\be=fr_U V_i$ 
 be an arc of circle of radius $r_0$ such that
 $r_0(1+\e)^n<\de$.
 
 Observe that $e$ is a fixed prime end of both $f$ and $f^n$.
 From Lemma~\ref{L3} we may take $i$ large enough so that 
 $f(\be)\cap \be\ne\emptyset$ and $f^n(\be)\cap\be\ne \emptyset$.
 
 We have that $f^{i-1}(\be)\cap f^i(\be)\ne\emptyset$ for $1\le i\le n$ 
 and $f^n(\be)\cap\be\ne \emptyset$.
 We are going to show that these arcs turn around $p$ at least
 once to close a curve with $p$ inside.
 
 Let $\hbe:=]a,b[\times\{r_0\}$ be a lifting of $\be$ with $0\le a<2\pi$.
 We have that $b-a<2\pi$.
 
 \begin{Lemma}\label{L6}\quad
 $F\hbe\cap\hbe\ne\emptyset$ 
 \quad or \quad
 $F\hbe\cap(\hbe+(2\pi,0))\ne\emptyset$.
 \end{Lemma}
 
 \begin{proof}
 Since $f\be\cap \be \ne \emptyset$, we have that 
 $F\hbe\cap (\hbe+(2k\pi,0))\ne\emptyset$ 
 for some $k\in\Z$. 
 We are going to show that $k=0$ or $1$.
 
 Let $(\th_1,r_1)\in F\hbe\cap (\hbe+(2k\pi,0))$.
 Then $(\th_1,r_1)=F(\th_0,r_0)$ and 
 $\th_1=\ov\th_0+2k\pi$,
 where both $\th_0$ and $\ov{\th_0}$ belong to $]a,b[$.
 By \eqref{e6} we have that 
 $|\th_1-(\th_0+\a)|<\tfrac \pi 2\e$.
 Therefore
 \begin{gather*}
 |\ov\th_0+2k\pi -(\th_0+\a)|<\tfrac\pi 2 \e,
 \\
 |2k\pi-\a|<|\th_0-\ov\th_0|+\tfrac\pi 2\e<2\pi+\tfrac\pi 2\e,
 \\
 \text{and}\qquad
 -2\pi+\a-\tfrac \pi 2\e
 <2k\pi
 <2\pi+\a+\tfrac\pi 2 \e.
 \end{gather*}
 By \eqref{e5} we have that 
 $-2\pi<2k\pi<4\pi$ and $k=0$ or $1$.
 
 \end{proof}
 
 \begin{Lemma}\label{L7}
 \quad
 $F^n\hbe\cap(\hbe+(2k\pi,0))\ne\emptyset$ \quad for some \quad $k\ge 1$.
 
 \end{Lemma}

\begin{proof}
\quad

Since $f^n\be\cap \be\ne\emptyset$ we have that 
$F^n\hbe\cap(\hbe+(2k\pi,0))\ne\emptyset$
for some $k\in \Z$.

Let $(\th_n,r_n)\in F^n\hbe\cap(\hbe+(2k\pi,0))$.
Then $(\th_n,r_n)=F^n(\th_0,r_0)$ 
and $\th_n=\ov\th_0+2k\pi$ 
where both
$\th_0$ and $\ov\th_0$ belong to $]a,b[$.
By~\eqref{e6}, we have that 
$|\th_n-(\th_0+n\a)|<n\frac\pi 2 \e$.
Therefore
\begin{gather*}
|\ov\th_0+2k\pi-(\th_0+n\a)|<n\tfrac\pi 2\e,
\\
2k\pi+\ov\th_0-\th_0-n\a>-n\tfrac\pi 2\e \qquad\text{and }
\\
2k\pi>\th_0-\ov\th_0+n\a-n\tfrac\pi 2 \e > -2\pi+n\a-n\tfrac \pi 2 \e>0
\quad\text{by }\eqref{e3}.
\end{gather*}
Hence $k\ge 1$.

\end{proof}

In the case when 
$F\hbe\cap\hbe\ne\emptyset$ we construct $\xi$ in the following way.

We have $F^{j+1}\hbe \cap F^j\hbe\ne\emptyset$ for $1\le j\le n-1$ 
and by Lemma~\ref{L7}, $F^n\hbe\cap(\hbe+(2k\pi,0))\ne\emptyset$ 
for some $k\ge1$.
This implies that the union 
$\hbe\cup F\hbe\cup\cdots\cup F^n\hbe\cup(\hbe+(2k\pi,0))$
is path connected and therefore this union contains a path connecting a
point $(\th_0,r_0)\in\hbe$ to $(\th_0+2k\pi,r_0)$ with $k\ge 1$.
By \eqref{e55} this path projects down to a closed path $\xi$
contained in $B^\circ_\de-\{0\}\subset B^\circ-\{0\}$
which is a non trivial element of $\pi_1(B^\circ-\{0\})$.
Since $\be\subset U$ and $\xi\subset \cup_{j=0}^n f^j\be$,
we have that $\xi\subset U$.
We have that $p$ and $\partial B$ 
lie in different components of $S-\xi$. 
Let $B'$ be the component of $S-\xi$ that contains $p$.
Then  $fr_SB'\subset \xi$. 
The set
$Z(b)$ intersects both $B'$ and $S-B'$ and therefore
$Z(b)$ intersects $\xi$.
But this is a contradiction since $Z(b)\subset fr_SU$ and 
$\xi\subset U$.

When $F\hbe\cap (\hbe+(2\pi,0))\ne\emptyset$ we work with a different
lifting of $f$, $$H(\th,r):=F(\th,r)-(2\pi,0).$$

\begin{Lemma}\label{L8}

$H$ satisfies $H\hbe\cap \hbe\ne\emptyset$
and 
$H^n\hbe\cap(\hbe+(2k\pi,0))\ne\emptyset$
for some $k\le -1$.

\end{Lemma}

\begin{proof}
The fact
$H\hbe\cap\hbe\ne\emptyset$ follows immediately from
the definition of $H$ and the fact that
$F\hbe\cap (\hbe+(2\pi,0))\ne\emptyset$.

Inequality
$f^n\be\cap \be\ne\emptyset$ implies that 
$H^n\hbe\cap (\hbe+(2k\pi,0))\ne\emptyset$
for some $k\in\Z$.

Let $q\in H^n\hbe\cap(\hbe+(2k\pi,0))$.
Since $H^n(\th_0,r_0)=F^n(\th_0,r_0)-(2n\pi,0)$,
we have that $q=F^n(\th_0,r_0)-(2n\pi,0)=(\th_n,r_n)-(2n\pi,0)$
and 
$q=(\ov\th_0+2k\pi,r_0)$ where both $\th_0$ and $\ov\th_0$ belong to $]a,b[$.
By~\eqref{e6} we have that
$|\th_n-(\th_0+n\a)|<n\frac \pi 2 \e$.
Therefore
\begin{gather*}
|\ov\th_0+2k\pi+2n\pi-\th_0-n\a|<n\tfrac\pi2\e
\\
2k\pi+\ov\th_0-\th_0+n(2\pi-\a)<n\tfrac \pi2\e
\\
2k\pi<\th_0-\ov\th_0-n(2\pi-\a)+n\tfrac \pi2\e
<2\pi-n(2\pi-\a)+n\tfrac \pi2\e<0
\qquad\text{by }\eqref{e4},
\end{gather*}
implying that $k\le -1$.

\end{proof}

Now we have $H\hbe\cap \hbe\ne \emptyset$ and 
$H^n\hbe\cap (\hbe+(2k\pi,0))\ne \emptyset$ 
for some $k\le -1$. 
The construction of $\xi$ is done as in the previous case.

Now we prove that one of the branches of $p$ 
is a connection contained in the frontier 
\linebreak
of $U$.

 \subsection{One of the branches of $p$ is a connection contained in the frontier of $U$.}
 \label{ss32}\quad

The arguments in this subsection are due to Mather~\cite{Mat9} and our presentation
is based on Franks, Le Calvez~\cite{FLC}, where they work with area preserving diffeomorphisms of the two
dimensional sphere and take powers of the map to assume that the branches of the fixed
point are invariant.

Firstly a well known result.

\begin{Lemma}\label{L9}
\quad

Let $K$ be a compact connected invariant set and $L$ a branch.
If $L\cap K\ne \emptyset$ then $L\subset K$.
\end{Lemma}

For a proof see Corollary~8.3 of Mather~\cite{Mat9}
or in $\SS^2$ Lemma~6.1 of Franks, Le Calvez~\cite{FLC}, or
\cite{OC2} for surfaces with boundary and partially defined 
area preserving maps.

  The domain $U$ has finitely many ideal boundary points and therefore $fr_SU$
  has finitely many connected components.
  Since $fr_SU$ is invariant, its components are periodic.
  Let $L$ be a branch of $p$. We claim that
  \begin{equation}\label{LU}
  L\cap U\ne\emptyset \qquad\then\qquad L\subset U.
  \end{equation}
  In fact, if $L$ were not contained in $U$
  then $L$ would intersect a component of $fr_SU$
  and by Lemma~\ref{L9} applied to a power of $f$
  this component would contain $L$.
  We have a little more:
  
  \begin{Lemma}\label{L10}
  \quad
  
  Let $L$ be a branch of $p$. Then we have the following:
  \begin{enumerate}%[(a)]
  \iitem\label{l10a} If $U$ is an open  invariant set and $L\cap U\ne\emptyset$
  then $U$ contains $L$ and the two sectors of $p$ adjacent to $L$.
  \iitem\label{l10b}
  If $K$ is a compact connected invariant set then either $L\subset K$
  or $L$ and its adjacent sectors are contained in one component of $S-K$.
  \end{enumerate}
  \end{Lemma}
  
  \begin{proof}\quad
  
  The branch $L$ is invariant or of period two. By~\eqref{LU} we have that $L\subset U$. Let $x\in L$ and let $W\subset U$ be
  a neighborhood of the arc from $x$ to $f^2(x)$ in $L$.
  Then $\cup_{n\in\Z}f^{2n}(W)$ contains the sectors of $p$ adjacent to $L$.
  Item~\eqref{l10b} easily follows from item~\eqref{l10a} letting $U$ be a component of $S-K$ which intersects $L$.
  
  \end{proof}

 Let $b=\a(e)$ and $C(b)$ be the circle of prime ends that contains $e$.
 
 Let $(x,y)$ be continuous coordinates in a neighborhood $V$ of $p$ with $p$
 at the origin where $f(x,y)=(\la x,\la^{-1}y)$ with $|\la|>1$.
 We may assume that $V$ is the open ball 
 $B^\circ_\de$ of radius $\de$ and center at $(0,0)$.
 In these coordinates the local branches and sectors are:
 \begin{alignat}{2}
 L_1&:=\{(x,y)\in B^\circ_\de\,|\, x>0,\,y=0\}
 ,\quad S_1&&:=\{(x,y)\in B^\circ_\de\,|\,x>0,\,y>0\},
 \notag\\
 L_2&:=\{(x,y)\in B^\circ_\de\,|\, x=0,\,y>0\}
 ,\quad S_2&&:=\{(x,y)\in B^\circ_\de\,|\,x<0,\,y>0\},
 \label{sectors}\\
 L_3&:=\{(x,y)\in B^\circ_\de\,|\, x<0,\,y=0\}
 ,\quad S_3&&:=\{(x,y)\in B^\circ_\de\,|\,x<0,\,y<0\},
 \notag\\
 L_4&:=\{(x,y)\in B^\circ_\de\,|\, x=0,\,y<0\}
 ,\quad S_4&&:=\{(x,y)\in B^\circ_\de\,|\,x>0,\,y<0\}.
\notag
\end{alignat}

By \eqref{accessible}, there exists a path $\be:]0,1]\to U$ such that $\lim_{t\to 0}\be(t)=p$
in $S$ and 
\linebreak
$\lim_{t\to 0}\beta(t)=e$ in $\hU$. We may assume that $\de$ 
is small enough so that $\be(1)\notin B^\circ_\de$ and we can apply
Lemma~\ref{L2}. We may also assume that $\be(]0,1[)\subset V$.

For every decreasing sequence $(r_n)$ contained in $]0,\de[$ with 
$\lim_{n\to\infty}r_n=0$ there exists a chain $(V_n)$ representing $e$
such that $\xi_n=fr_UV_n\subset C_{r_n}$.
Notice that $\xi_n\cap \be\ne\emptyset$ for every $n$.
Since $\be(]0,1[)\cup (\cup_n\xi_n)$ is connected an contained in both
$U$ and $V$, all the arcs $\xi_n$ must be contained in one connected 
component of $U\cap V$.
Denote this component by $W$.

\begin{Lemma}\label{L11}
\quad

There exists an open arc of prime ends $]a,c[$ containing $e$ in $C(b)$
such that the function $\phi:]a,c[\to fr_SU$ defined by $\phi(e')=X(e')$ is 
continuous, $\phi(]a,e[)$ and $\phi(]e,c[)$ are local branches of $p$
 and $\phi\circ \hf=f\circ\phi$.
\end{Lemma}

\begin{proof}
\quad

First we consider the case when the arcs $\xi_n$ do not intersect the local
branches of $p$. This implies that $W$ must be contained in one of the sectors
of $p$ defined by $B^\circ_\de$, say $S_1$.

By Lemma~\ref{L3}, if $n_0=\inf\{n\,|\,f(V_n)\subset V_1\}$, then
$f(\xi_n)\cap \xi_n\ne\emptyset$ for $n\ge n_0$.
This implies that $\la>0$. Let $\rho$ be any number in $]0,r_{n_0}[$. 
By \eqref{l2b} of  Lemma~\ref{L2} we may assume that 
$\xi_m\subset C_\rho$ for some $m>n_0$.

Let us consider the sup norm on $\re^2$. In this case $B_\rho$ is the square
with vertices at $(\pm\de,\pm\de)$, and if
$\Ga_\rho =(]0,\rho]\times\{\rho\})\cup(\{\rho\}\times]0,\rho])$
then $\xi_m\subset \Ga_\rho$.
Since $f(\Ga_\rho)\cap \Ga_\rho=\{(\rho,\la^{-1}\rho)\}$, we have that 
$f(\xi_m)\cap \xi_m=\{(\rho,\la^{-1}\rho)\}$ as well.
Therefore $(\rho,\la^{-1}\rho)$ and $(\la^{-1}\rho,\rho)$ belong to $\xi_m$
and the arc $\La_\rho$ from $(\rho,\la^{-1}\rho)$ to $(\la^{-1}\rho,\rho)$
inside $\Ga_\rho$ is contained in $\xi_m\subset U$.
This holds for every $\rho\in]0,r_{n_0}[$.
It follows that $\cup_{k\in\Z}f^k(\cup_{\rho\in]0,r_{n_0}[}\La_\rho)$
is contained in $U$ and contains $S_1\cap B^\circ_\ve$ for some
small number $\ve<\de$.

Therefore $S_1\cap B^\circ_\ve\subset W$ and $(L_1\cup L_2)\cap U=\emptyset$.
It follows that for $n$ sufficiently large we have that 
$V_n\subset S_1\cap B^\circ_\ve$ and the arcs $\xi_n$ must have one 
end point at $\{0\}\times ]0,\ve[$ and another at $]0,\ve[\times\{0\}$.
For any $x\in(\{0\}\times[0,\ve[)\cup([0,\ve[\times\{0\})$
if $B^\circ_{1/n}(x)$ is the open ball of radius $\tfrac 1n$ with center
at $x$, then $V_n=B^\circ_{1/n}(x)\cap W$ defines a prime chain
$(V_n)=e'$ such that $X(e')=\{x\}$.
Then there exists an open arc of prime ends $]a,c[$ containing $e$
in $C(b)$ such that if $e'\in]a,c[$ then $X(e')=\{x\}$ for some
$x=x(e')\in(\{0\}\times[0,\ve[)\cup([0,\ve[\times\{0\})$.
If we define $\phi:]a,c[\to S$ by $\phi(e')=x(e')$, then it is easy to check that
$\phi$ is continuous, $\phi(e)=p$, $\phi(]a,e[)=\{0\}\times[0,\ve[\subset L_2$
and $\phi(]e,c[)=[0,\ve[\times\{0\}\subset L_1$. Since
$X(\hf(e'))=f(X(e'))$ for any prime end $e'$ we have that 
$\phi\circ\hf=f\circ\phi$.

Now consider the case when the arcs $\xi_n$ intersect the local branches $L_i$.
The family $\{\xi_n\}$ can not intersect the four branches, otherwise by
 Lemma~\ref{L10}.\eqref{l10a}, $W$ would contain the four sectors of $p$ 
 implying that $B^\circ_\ve\cap fr_SU=\{p\}$ for small $\ve$,
 contradicting the fact that $b$ is regular.
 
 So $\{\xi_n\}$ intersects at least one of the local branches and at least
 one is disjoint from $\{\xi_n\}$. 
 We may assume that $L_1\cap\{\xi_n\}=\emptyset$ and
 $L_2\cap\{\xi_n\}\ne\emptyset$.
 
 Let $k=\sup\{j\in\{2,3,4\}\,|\,L_j\cap\{\xi_n\}\ne\emptyset\}$.
 Then by Lemma~\ref{L10}.\eqref{l10a} we have that 
 $L_1\cap W=\emptyset$, $L_{k+1}\cap W=\emptyset$ ($L_5=L_1$) and
 $$
 (S_1\cup\cup_{i=2}^k(L_i\cup S_i))\cap B^\circ_\ve
 \subset W\subset
 S_1\cup \cup_{i=2}^k(L_i\cup S_i)
 \quad
 \text{for some}\quad \ve\in]0,\de[.
 $$
 
 It follows that for $n$ sufficiently large we have that 
 $V_n\subset (S_1\cup\cup_{i=2}^k(L_i\cup S_i))\cap B^\circ_\e$.
 The sets $L_1\cap B^\circ_\ve$ and $L_{k+1}\cap B^\circ_\ve$ are contained in $fr_SW$
 and the arcs $\xi_n$ must have one end point at 
 $L_1\cap B^\circ_\e$ and another at $L_{k+1}\cap B^\circ_\ve$.
 
 As in the precious case, if $x\in(L_1\cup L_{k+1})\cap B^\circ_\e$ and 
 $V_n= B^\circ_{1/n}(x)\cap W$ then $\phi((V_n))=x$ is the desired map.
 In case $k=4$, $\phi(]a,c[)\subset\{p\}\cup L_1$.
 
 \qed

 Let $]e_1,e_2[$ be an arc of prime ends with $\hf(e_i)=e_i$ and 
 $$\lim\limits_{n\to-\infty}\hf^n(e)=e_1 \quad \text{and}\quad 
 \lim\limits_{n\to+ \infty}\hf^n(e)=e_2
 \quad \text{for any } e\in]e_1,e_2[
 $$
(possibly with $e_1=e_2$).
From Lemma~\ref{L11} we know that there exists a function
$\phi_1$ that maps an arc $]e_1,c[$ onto a local branch  
of $p_1=\phi_1(e_1)$ which is part of an unstable branch $L_1$ of $p_1$.
For any $e\in]e_1,e_2[$ there exist $n\in\Z$ and $e'\in]e_1,c[$
such that $e=\hf^n(e')$. Since $X(\hf^n(e'))=f^n(X(e'))$
we have that $e$ is accessible and $\phi_1(e)=X(e)$ extends
$\phi_1$ to $]e_1,e_2[$ with the same properties.
We have that $L_1=\phi_1(]e_1,e_2[)$.

The same happens with $e_2$. There exists a function $\phi_2$ that maps
and arc $]a,e_2[$ onto a local branch of $p_2=\phi_2(e_2)$ which is 
part of a stable branch $L_2$ of $p_2$. The map $\phi_2(e)=X(e)$ extends to
$]e_1,e_2[$ in the same way and $L_2=\phi_2(]e_1,e_2[)$.
Therefore $L_1=\phi_1(]e_1,e_2[)=\phi_2(]e_1,e_2[)=L_2$ is
a connection contained in $fr_SU$.

So the proof of item~\eqref{t21} of Theorem~\ref{T2} is complete.

 \end{proof}

\subsection{Proof of items \eqref{t22}, \eqref{t23} and \eqref{t24} of Theorem~\ref{T2}.}
\label{ss33}
\quad

Let $L$ be an invariant branch of a fixed point $p$ of saddle type and assume that all fixed
points of $f$ contained in $cl_SL$ are non degenerate. We are going to assume that $L$ does not
accumulate on one of its adjacent branches through the adjacent sector and show that $L$ must be a connection.

As in the previous subsection we are going to consider a system $(x,y)$ of continuous coordinates
in a neighborhood $V$ of $p$ with $p$ at the origin where $f(x,y)=(\la x ,\la^{-1}y)$ with $\la>1$.
We may assume that $V$ is the open ball $B^\circ_\de$ of radius $\de$ and center at $(0,0)$.
We are going to assume that $L$ is the branch that contains $\{(x,y)\in B^\circ_\de\,|\,x>0,\,y=0\}$
and that $L$ is disjoint from $S_1=\{(x,y)\in B^\circ_\de\,|\,x>0,\,y>0\}$. 
Let $U$ be the connected component of $S-cl_SL$ that contains $S_1$.
Then $f(U)=U$ and by Lemma~2.3 in \cite{Mat9} or Proposition~3.15 in \cite{OC2}, $U$ has finitely many ideal boundary points
(in fact $U$ has at most $g+1$ ideal boundary points where $g$ is the genus of $S$, but we do not
need this here). We also have that $fr_SU\subset cl_SL$ and thus $fr_SU$ contains 
no degenerate fixed point of $f$. Therefore $U$ satisfies the hypothesis of item~\eqref{t21}.

We claim that  $\{p\}$ is the principal set of a fixed prime end $e$. When $L_2\cap U=\emptyset$ 
we see this by defining the chain $V_n=B^\circ_{1/n}\cap S_1$. When $L_2\subset U$ 
and $L_3\cap U=\emptyset$ we take $V_n$ as $B^\circ_{1/n}\cap (S_1\cup L_2\cup S_2)$, when 
$L_3\subset U$ and $L_4\cap U=\emptyset$ we take $V_n$ as
$B^\circ_{1/n}-(L_4\cup S_4\cup L_1)$ and when $L_4 \subset U$ we take  $V_n$ as $B^\circ_{1/n}-L_1$.

From Lemma~\ref{L11} there exists an open arc of prime ends $]a,c[$  in $C(b)$ containing $e$
such that the function $\phi:]a,c[\to fr_SU$, defined by $\phi(e')=X(e')$, is continuous
and $\phi(]a,e[)$ and $\phi(]e,c[)$ are local branches of $p$.
The local branch of $L$ is equal to either $\phi(]a,e[)$ or $\phi(]e,c[)$ or both.
In the proof of item~\eqref{t21} we showed that both $\phi(]a,e[)$ 
and $\phi(]e,c[)$ are part of connections.
Therefore $L$ must be a connection.

We proved that $L$ accumulates on both adjacent sectors.
From  Lemma~\ref{L10}.\eqref{l10b}
we have that if $L$ accumulates on one adjacent sector then 
$L\subset \om(L)$. This completes the proof of item~\eqref{t22}.

Assume that $L_1$ and $L_2$ are invariant adjacent branches that are not connections 
and that all fixed points of $f$ contained in $cl_S(L_1\cup L_2)$ are non degenerate.
From Lemma~\ref{L10}.\eqref{l10b} we have that
 $L_2\subset cl_S L_1$ and $L_1\subset cl_SL_2$ which proves 
 item \eqref{t23}.
 
 Item~\eqref{t24} follows from item~\eqref{t23}.

\section{Surfaces with boundary and homoclinic points.}
\label{S4}

In this section we  extend previous results to
compact connected orientable surfaces $S$ with boundary.

Let $f:S\to S$ be an orientation preserving area preserving homeomorphism of $S$.
Let $p\in\partial S$ be a periodic point of $f$ of period $\tau$.
We say that $p$ is of saddle type if there exist a neighborhood
$V$ of $p$ and continuous coordinates $(x,y)$, $y\ge 0$,
in $V$ with $p$ at the origin and $f^\tau(x,y)=(\la x,\la^{-1}y)$ 
with $\la>0$ and $\la\ne 1$. The point $p$ has two sectors and 
three branches, two of which are connections contained in $\partial S$
and the other is contained in $S-\partial S$.

 Let $C$ be a connected component of $\partial S$.
 Suppose that all periodic points in $C$ are of saddle type.
Then there are only finitely many periodic points in $C$ and, as we move around $C$,
periodic points must be alternatively attracting and repelling in $C$
and their branches that are contained in $S-\partial S$ also alternate from stable 
to unstable.

\begin{Theorem}\label{T6}\quad

Let $S$ be a compact connected orientable surface with boundary
provided with a finite measure $\mu$ which is positive on open sets 
and $f:S\to S$ be an orientation preserving area preserving homeomorphism of $S$.
\begin{enumerate}
\iitem\label{t61} 
Suppose that $L$ is an invariant branch of $f$ and that all fixed points of $f$
contained in $cl_SL$ are non degenerate. Then either $L$ is a connection or
$L$ accumulates on both adjacent sectors. In the later alternative $L\subset \om(L)$.

\iitem\label{t62}
Let $p\in S-\partial S$ be a fixed point of $f$ of saddle type and let  $L_1$ and $L_2$
be adjacent branches of $p$ that are not connections.
If $L_1$ and $L_2$ are invariant and all fixed points of $f$ contained in $cl_S(L_1\cup L_2)$
are non degenerate then $cl_SL_1=cl_S L_2$. If in addition $S$ has genus 0
then $L_1\cap L_2\ne \emptyset$.

\iitem\label{t63}
Let $p\in S-\partial S$ be a fixed point of $f$ of saddle type. 
Assume that all fixed points of $f$ contained in 
$cl_S(W^u_p\cup W^s_p)$ are non degenerate.
If the branches of $p$ are invariant and none of them is a connection,
then all branches have the same closure in $S$. If in addition the genus of $S$ 
is $0$ or $1$, then the four branches of $p$ have homoclinic points.

\iitem\label{t64}
Let $C$ be a connected component of $\partial S$ and suppose that 
all fixed points $p_1, \ldots, p_{2n}$ of $f$ in $C$ are of saddle type.
Let $L_i$ be the branch of $p_i$ contained in $S-\partial S$.
Assume that for every $i$ all fixed points of $f$ contained in $cl_SL_i$
are non degenerate and that $L_i$ is not a connection.
Then $cl_SL_i=cl_SL_j$ for any pair $(i,j)$.

If in addition $S$ has genus $0$ then any pair $(L_i,L_j)$ of stable and unstable
branches intersect. The same happens when the genus of $S$ is $1$
provided there are at least $4$ fixed points in $C$.

\end{enumerate}

\end{Theorem}

\begin{Remark}\quad

The conclusion about the existence of homoclinic and heteroclinic points
above is false if the genus of $S$ is greater than one. Examples could be 
time one maps of area preserving flows on surfaces with finitely many 
singularities and every other orbit dense. In item~\eqref{t62} the conclusion
about homoclinic points is false if the genus of $S$ is one.
In the torus two branches could close a connection and the other 
two spin around the torus like a line of irrational slope without intersecting.
For the same reason,
when the genus of $S$ is one in item~\eqref{t64},
we need at least four recurrent branches in a boundary component
to have heteroclinic points.
\end{Remark}

For applications we need to state a theorem obtainnig homoclinics 
for area preserving maps defined on an open set of a surface with boundary.
So we split Theorem~\ref{T6} into two theorems, one containing 
the results in accumulation of invariant manifolds and another
containing the results on homoclinics.

\begin{Theorem}\quad\label{T7}

Let $S$ be a compact connected orientable surface with boundary
provided with a finite measure $\mu$ which is positive on open sets 
and $f:S\to S$ be an orientation preserving area preserving homeomorphism of $S$.
\begin{enumerate}
\iitem\label{t71} 
Suppose that $L$ is an invariant branch of $f$ and that all fixed points of $f$
contained in $cl_SL$ are non degenerate. Then either $L$ is a connection or
$L$ accumulates on both adjacent sectors. In the later alternative $L\subset \om(L)$.

\iitem\label{t72}
Let $p\in S-\partial S$ be a fixed point of $f$ of saddle type and let  $L_1$ and $L_2$
be adjacent branches of $p$ that are not connections.
If $L_1$ and $L_2$ are invariant and all fixed points of $f$ contained in $cl_S(L_1\cup L_2)$
are non degenerate then $cl_SL_1=cl_S L_2$. 

\iitem\label{t73}
Let $p\in S-\partial S$ be a fixed point of $f$ of saddle type. 
Assume that all fixed points of $f$ contained in 
$cl_S(W^u_p\cup W^s_p)$ are non degenerate.
If the branches of $p$ are invariant and none of them is a connection,
then all branches have the same closure in $S$.

\iitem\label{t74}
Let $C$ be a connected component of $\partial S$ and suppose that 
all fixed points $p_1, \ldots, p_{2n}$ of $f$ in $C$ are of saddle type.
Let $L_i$ be the branch of $p_i$ contained in $S-\partial S$.
Assume that for every $i$ all fixed points of $f$ contained in $cl_SL_i$
are non degenerate and that $L_i$ is not a connection.
Then $cl_SL_i=cl_SL_j$ for any pair $(i,j)$. 

\end{enumerate}
\end{Theorem}

\begin{Theorem}\quad\label{T8}

Let $S$ be a compact connected orientable surface with boundary.
Let $S_0\subset S$ be a submanifold with compact boundary 
$\partial S_0\subset \partial S$
 and let $f,f^{-1}:S_0\to S$
be an orientation preserving and area preserving homeomorphism of $S_0$ onto
open subsets $fS_0$, $f^{-1}S_0$ of $S$ with $f(\partial S_0)\subset \partial S_0$.
\begin{enumerate}
\iitem\label{t81}
Let $p\in S_0-\partial S$ be a fixed point of $f$ of saddle type
and let $L_1$, $L_2$ be adjacent and invariant branches of $p$ such that 
$cl_S L_i\subset S_0$, $i=1,2$ and that both branches accumulate 
on the sector bounded by them.
If in addition $S$ has genus 0, 
then $L_1\cap L_2\ne\emptyset$.

\iitem\label{t82}
Let $p\in S_0-\partial S$ be a fixed point of $f$ of saddle type.
Assume that the branches of $p$ are invariant, with closure in $S$
included in $S_0$.
Assume also that each branch of $p$ is not a connection, accumulates 
of its adjacent sectors and that all the branches of $p$
have the same closure in $S$.
If in addition $S$ has genus 0 or 1,
then the four branches of $p$ have homoclinic points.

\iitem\label{t83} 
Let $C$ be a connected component of $\partial S_0$ and suppose that 
all fixed points $p_1,\ldots,p_{2n}$ of $f$ in $C$
are of saddle type.
Let $L_i$ be the branch of $p_i$ contained in
$S-\partial S$.
Assume that for every $i$, $L_i$ is not a connection, $L_i$ accumulates 
on  both of its adjacent sectors and that
$cl_SL_i=cl_SL_j\subset S_0$ for every pair $(i,j)$.

If in addition $S$ has genus 0, then every pair $(L_i, L_j)$ of stable and 
unstable branches intersect. 
The same happens if the genus of $S$
is 1 provided that there are at least $4$ fixed points in $C$.
\end{enumerate}

\end{Theorem}

\subsection{The double of a surface.}
\label{ss41}
\quad

In order to prove Theorem~\ref{T6} we would like to use 
the previous results for surfaces without boundary.
For this we are going to work with the double of $S$.

Let $\sim$ be the equivalence relation defined by the 
partition of $S\times \{0,1\}$ into one point sets
$\{(p,i)\}$ if $p\notin\partial S$, $i\in\{0,1\}$ and two point sets 
$\{(p,0),(p,1)\}$ if $p\in \partial S$.
Let $S'=S\times\{0,1\}/\sim$ equipped with the quotient topology.
The surface $S'$ is obtained by gluing two copies of $S$ together along their
common boundary and it is called the {\it double} of $S$.
The quotient map $S\times\{0,1\}\to S'$ is closed implying that $S'$ 
is a Hausdorff space.
If $S$ is compact, connected and orientable, then so is $S'$.
The mapping $\iota:S\to S'$ defined by $\iota(p)=[(p,0)]$ 
is a homeomorphism from $S$ onto a closed subset $\iota(S)$
of $S'$. So we just think of $S$ as a subset of $S'$.
Since $\iota(S)$ is closed in $S$ we have that if $A\subset S$
then $cl_SA=cl_{S'}A$.

The measure $\mu$ can be ``doubled'' to a measure $\mu'$ on $S'$
by taking the measure $\nu(\{0\})=\nu(\{1\})=1$ on $\{0,1\}$,
the product measure $\mu\times \nu$ on $S\times\{0,1\}$
and the pushforward by the quotient map.
Homeomorphisms of $S$ extend naturally to homeomorphisms
of $S'$ by defining $f'[(x,i)]=[(f(x),i)]$. If $f$ is orientation
preserving and area preserving, the so is $f'$.

The set $S'$ has the structure of a surface without boundary.
Charts for neighborhoods of points $[(p,i)]\in S'$ with 
$p\in\partial S$ may be obtained in the following way.
Let $B$ be an open ball with center at $(0,0)$,
$B^+=\{(x,y)\in B\,|\,y\ge 0\}$ and 
$\phi:B^+\to S$ a chart for a point $p=\phi(0,0)\in\partial S$.
Define $\phi':B\to S'$ by $\phi'(x,y)=[(\phi(x,y),0)]$
if $y\ge 0$ and $\phi'(x,y)=[\phi(x,-y),1)]$ if $y\le 0$.
It follows from our definitions that if $p$ is a saddle in $\partial S$
then it is a saddle in $S'$.

\subsection{Proof of Theorem~{\ref{T7}}: the closure of invariant manifolds.}

\subsection{Proof of item~{\eqref{t71}}: auto accumulation of invariant manifolds.}
\label{ss43}
\quad

Let $L$ be an invariant branch of $f$ such that all fixed points of $f$
contained in $cl_SL$ are non degenerate. 
Let $S'$ be the double of $S$ and $f'$ the extension of $f$ to $S'$.
As a branch of $f'$, $L$ is  invariant by $f'$ and
all fixed points of $f'$ contained in $cl_{S'}L$ are non degenerate.
By item~\eqref{t22} of Theorem~\ref{T2},
item~\eqref{t71} holds.

\subsection{Proof of items~\eqref{t72} and \eqref{t73}: equal closure of invariant manifolds.}
\label{ss44}
\quad

Now let $p\in S -\partial S$ be a fixed point of $f$ of saddle type
and let $L_1$ and $L_2$ be invariant adjacent branches of $p$ that are not 
connections and assume that all fixed points of $f$ contained in 
$cl_S(L_1\cup L_2)$ are non degenerate. By item~\eqref{t23}
of Theorem~\ref{T2} applied to $f'$ we have that $cl_SL_1=cl_S L_2$.

That all the branches in item~\eqref{t73} have the same closure 
follows immediately from item~\eqref{t72}.

\subsection{Proof of item~\eqref{t74}: the case of  boundary fixed points. }
\label{ss45}
\quad

Let $C$ be a connected component of $\partial S$ and suppose that 
all fixed points of $f$ in $C$ are of saddle type.
Let $p_1,\ldots,p_{2n}$ be these fixed points and $L_1,\ldots,L_{2n}$
their branches contained in $S-\partial S$. We are assuming that for
every $i$ all fixed points of $f$ contained in $cl_SL_i$ are non degenerate
and that $L_i$ is not a connection.

We need to show that all these branches have the same closure, and for this
it is enough to show that this happens for any pair of branches whose
corresponding fixed points have a connection contained in $C$ as a common branch.

\begin{Lemma}\label{L12}
Let $L_1$ and $L_2$ be any branches of $f$ possibly from different periodic points in $S$.
If $L_1$ accumulates on a sector adjacent to $L_2$ then $L_2\subset cl_SL_1$
(for surfaces with or without boundary).
\end{Lemma}

\begin{proof}\quad

Just take $K=cl_{S'}L_1$ and $L=L_2$ in part~\eqref{l10b} in Lemma~\ref{L10}.

\end{proof}

Let $p_1$ and $p_2$ be fixed points of $f$ in $C$, and  let $L_1$ and $L_2$ be their branches
contained in $S-\partial S$.
It suffices to consider the case when there is a connection $C_{12}$ from $p_1$ to $p_2$.
By item~\eqref{t71} of Theorem~\ref{T7}, the branch $L_1$ accumulates on both sectors of $p_1$ adjacent to itself and therefore
by Lemma~\ref{L12}, we have that $C_{12}\subset cl_SL_1$. From this we have that 
$L_1$ accumulates on the sector of $p_2$ that has $C_{12}$ and $L_2$ as adjacent branches.
By Lemma~\ref{L12} again we have that $L_2\subset cl_SL_1$. It follows that $cl_SL_1=cl_SL_2$.

\subsection{Proof of Theorem~\ref{T8}:  homoclinics.}
\quad

\subsection{Proof of item~\eqref{t81}: homoclinics for specified branches in genus 0.}
\label{ssp81}
\quad

Assume now that $S$ has genus $0$.
Replacing $f$ by $f^2$ if necessary, we can assume that the branches 
$L_1$ and $L_2$ are fixed.

Consider a system of coordinates
in a neighborhood $V$ of $p$ with $p$
at the origin and in which $f(x,y)=(\la^{-1}x,\la y)$ with 
$0<\la<1$, $y\ge 0$.
We assume that $L_1$ is the unstable branch that 
contains $\{(x,y)\in V\,|\, y=0,\,x>0\}$ and $L_2$ is 
the stable branch that contains $\{(x,y)\in V\,|\,x=0,\,y>0\}$.
Since the branches $L_1$, $L_2$ are invariant and contained in $S_0$, we have that 
all their iterates by $f$ and $f^{-1}$ are defined, i.e. 
$$
\forall n\in \Z \quad \forall i\in\{1,2\} \quad f^n(L_i)\subset S_0.
$$

By hypothesis
the branches $L_1$ and $L_2$ accumulate on the sector
\linebreak
$\{(x,y)\in V\,|\, x>0,\,y>0\}.$ Let
\begin{gather}
\Si=\{(x,y)\in V\,|\, 0<xy\le \la^2,\, 0<x\le 1,\,0<y\le 1\},
\notag
\\
En(\Si)=\{(x,y)\in\Si\;|\; \la<y\le 1\},
\label{ens}
\\
Ex(\Si)=\{(x,y)\in\Si\;|\; \la<x\le 1\}.
\notag
\end{gather}

Note that $En(\Si)$ and $Ex(\Si)$ are disjoint.
From the dynamics of $f$ in $\Si$ we see
that every orbit of $f$ that visits $\Si$ must enter 
$\Si$ through $En(\Si)$, meaning that if
$x\notin\Si$ and $f^n(x)\in\Si$ for some positive
$n$, then for the smallest $n$ with this property
$f^n(x)\in En(\Si)$. In the same way every orbit
of $f^{-1}$ that visits $\Si$ must enter $\Si$
through $Ex(\Si)$.

We start at $p$ and move along the unstable 
branch $L_1$, let $q_1$ be the first point of $L_1$
to intersect $\Si$. Since $L_1$ is unstable we have that 
$q_1\in En(\Si)$. Join $q_1$ and $p$ by a line segment
$\ga_1\subset int_S\Si$ and consider 
$\Ga_1=L_1[p,q_1]\cup \ga_1$, where $L_1[p,q_1]$
is the segment of $L_1$ between $p$ and $q_1$.
Observe that 
$\Ga_1$ is a simple closed curve that separates $S$
into two connected open sets whose common frontier in $S$ is $\Ga_1$.

If we start at $p$ and move along the stable branch $L_2$,
let $q_2$ be the first point of $L_2$ to intersect $\Si$.
Since $L_2$ is stable we have that $q_2\in Ex(\Si)$.
Join $q_2$ and $p$ by a line segment $\ga_2\subset int_S\Si$
and consider $\Ga_2=L_2[p,q_2]\cup \ga_2$, where 
$L_2[p,q_2]$ is the segment of $L_2$ between $p$ and $q_2$.
Then
$\Ga_2$ is also a simple closed curve in $S$.
Note that the local branch of $L_2$ and $\ga_2$ are in different components
of $S-\Ga_1$. From this we conclude that $\Ga_1$ and $\Ga_2$
must intersect again at a point $q$ different from $p$.
Since $\Ga_1\cap\ga_2=\emptyset$ and $\Ga_2\cap \ga_1=\emptyset$
we have that $q\in L_1[p,q_1]\cap L_2[p,q_2]$.

\subsection{Proof of item~\eqref{t82}: homoclinics in genus 0 and 1.}
\label{ssp82}
\quad

The existence of  homoclinic points if the genus of $S$ is $0$
follows immediately from item~\eqref{t81}.
Consider now the case when $S$ has genus $1$.

We may assume that $S$ is the two dimensional torus $\T^2$ minus a finite union
of (sets homeomorphic to) open disks. By restricting the canonical 3-fold covering
\begin{equation}\label{3fold}
\re^2/(3\Z\times\Z)\to\re^2/\Z^2
\end{equation}
 we obtain a 3-fold covering 
$\pi:E\to S$ whose group of deck transformations is generated by the rotation
$T(x,y)=(x+1,y)$. Let $F:E\to E$ be a lifting of $f$ such that the three points
in $\pi^{-1}\{p\}$ are fixed.

Let $\pi^{-1}(p)=\{p_1,p_2,p_3\}$. If $L$ is a branch of $p$ then 
$\pi^{-1}(L)=L_1\cup L_2\cup L_3$ where every $L_i$ is a branch of $p_i$,
$L_2=TL_1$ and $L_3=TL_2$. In order to prove that $p$ has homoclinic
points, it is enough to prove that for some pair $(i,j)$ an unstable branch 
of $p_i$ intersects a stable branch of $p_j$. The branches of $p$ have the 
same closure in $S$. We prove in Lemma~\ref{l47} that the same happens with the points $p_i$.

 Let $K_i$
be the closure in $E$ of the branches of each $p_i$.

The following version of the accumulation lemma is proven in Theorem~4.3 of~\cite{OC2}.

\begin{Lemma} {\cite[Th. 4.3]{OC2}}\label{acclem2}\quad

 Let $S$ be a  connected surface with compact boundary provided with 
a Borel measure $\mu$ such that open non-empty subsets have positive
measure and compact subsets have finite measure.
Let $S_0\subset S$ be an open subset with $fr_SS_0$ compact.

%Let $S_0\subset S$ be a submanifold such that  $fr_S S_0$ is compact and 
%whose boundary is complete and
%satisfies $\partial S_0\subset \partial S$.

Let $f,f^{-1}:S_0\to S$ be an area preserving homeomorphism of $S_0$
onto an open subsets $f(S_0)$, $f^{-1}(S_0)$ of $S$.
Let $K\subset S_0$ be a compact connected invariant subset of $S_0$.

If $L\subset S_0$ is a branch of $f$ and $L\cap K\ne \emptyset$,
then $L\subset K$.

\end{Lemma}

\begin{Lemma}\label{l47}
\quad

All the branches of each $p_i$ have the same closure $K_i$.
There are two disjoint alternatives:
\begin{enumerate}
\iitem\label{t631}
For every pair $(i,j)$ all branches of $p_i$ accumulate 
on all sectors of $p_j$ and $K_i=K_j$.
\iitem\label{t632}
There exist pairwise disjoint neighborhoods $V_i$ of $p_i$ such that 
$K_i\cap V_j=\emptyset$ if $i\ne j$.
\end{enumerate}
Alternative~\eqref{t631} holds when there exists a pair $(i,j)$ with $i\ne j$ for which a branch 
of $p_i$ accumulates on a sector of $p_j$.
\end{Lemma}

In the case~\eqref{t632} it might happen that $K_i\cap K_j\ne\emptyset$, but we do not need this here.

\begin{proof}\quad

Let $L^i_\a$, $1\le\a\le4$ be the branches of $p_i$ and let $S^i_\a$ be the sectors of $p_i$ with 
\linebreak
$(L^i_\a\cup L^i_{\a+1})\cap B_\de\subset cl_{E}(S^i_\a)$ as in \eqref{sectors}.

We first prove that in alternative \eqref{t632} all the branches of $p_i$
have the same closure. We use alternative \eqref{t632} with $K_i$
being the closure of $\cup_{\a=1}^4L^i_\a$.  By hypothesis \eqref{t82},
 in $S$ all the branches of $p$ have the same closure.
 If $1\le \a,\be\le 4$ the branch $L^i_\a$ accumulates on some branch $L^j_\be$,
 i.e. $L^j_\be\cap cl_E(L^i_\a)\ne \emptyset$.
 The statement of alternative~\eqref{t632} implies that $j=i$.
 By Lemma~\ref{acclem2}, 
$L^i_\be\subset cl_E(L^i_\a)$.
Similarly $L^i_\a\subset cl_E(L^i_\be)$.
Thus $cl(L^i_\a)=cl_E(L^i_\be)= K_i$.

\begin{Claim}\label{cl47}\quad

If there is $\a$ and $i\ne j$ such that $L^i_\a$ accumulates on $L^j_\a$ 
then alternative~\eqref{t631} holds.
\end{Claim}
For simplicity assume that $\a=1$, $(i,j)=(1,2)$.
We have that $L^2_1\cap cl_E L^1_1\ne \emptyset$.
By Lemma~\ref{acclem2}, $L^2_1\subset cl_E(L^1_1)$.
Then $L^3_1=TL^2_1\subset Tcl_E(L^1_1)=cl_E(L^2_1)$.
Similarly
\begin{equation}\label{incL}
L^1_1\subset cl_E(L^3_1)\subset cl_E(L^2_1)\subset cl_E(L^1_1)=:K,
\end{equation}
and the inclusions in \eqref{incL} are equalities.
Now let $\be\ne 1$ and $k\in\{1,2,3\}$.
We have that some $L^j_1$ accumulates on $L^k_\be$ and hence 
$L^k_\be\subset cl_E(L^j_1)=K$. Also $L^k_\be$ accumulates on some 
$L^i_1$ and then $K=cl_E(L^i_1)\subset L^k_\be$.
This proves the claim.

The hypothesis in~\eqref{t82} implies that any branch of $p$ accumulates on all the sectors of $p$.
By Lemma~\ref{acclem2} if a branch $L^i_\a$ accumulates on a sector $S^j_\be$ then 
$L^j_\be\cup L^j_{\be+1}\subset cl_E(L^i_\a)$.

Suppose that some $L^i_\a$ accumulates on the sector $S^j_\be$ with $j\ne i$.
Alternative \eqref{t632} is just the negation of this assumption.
So it remains to prove that alternative \eqref{t631} holds.
Assume by contradiction that alternative \eqref{t631} does not hold.
For simplicity $(i,\a)=(1,1)$.
Then
\begin{equation}\label{Ljb+1}
L^j_\be\cup L^j_{\be+1}\subset cl_E(L^1_1).
\end{equation}
For some $k$ the branch $L^j_\be$ accumulates on $L^k_1$ and then
$L^k_1\subset cl_E(L^j_\be)\subset cl_E(L^1_1)$.
If $k\ne 1$ claim~\ref{cl47} implies that alternative~\eqref{t631} holds.
Therefore $k=1$. Hence $cl_E(L^j_\be)=cl_E(L^1_1)$.
Similarly 
$$
cl_E(L^j_{\be+1})=cl_E(L^j_\be)=cl_S(L^1_1).
$$

The branch $L^j_{\be+1}$ accumulates on some sector $S^k_{\be+1}$ and hence
\begin{equation}\label{Lkb+1}
L^k_{\be+1}\cup L^k_{\be+2}\subset cl_E(L^j_{\be+1})=cl(L^1_1).
\end{equation}
By claim~\ref{cl47} applied to $\a=\be+1$ we have that $k=j$.
The same argument as above using~\eqref{Lkb+1} instead of \eqref{Ljb+1} gives
$$
cl_E(L^j_{\be+2})=cl_E(L^j_{\be+1})=cl_E(L^j_\be)=cl_E(L^1_1).
$$
Repeating the argument once more we get that $cl_E(L^j_\be)=cl_E(L^1_1)$
for all $\be$.
In particular for $\be=1$, $cl_E(L^j_1)=cl_E(L^1_1)$.
Claim~\ref{cl47} for $\a=1$  implies that alternative~\eqref{t631} holds.
A contradiction.

\end{proof}

Assume that alternative~\eqref{t631} holds. 

Choose a sector $\Si$ of some $p_i$, say $p_1$, and continuous coordinates about $p_1$
with $p_1$ at $(0,0)$ such that $f(x,y)=(\la^{-1}x,\la y)$, $0<\la<1$. As before we will write $\Si$ as
$$
\Si=\{(x,y)\in V\,|\, 0<xy\le\la^2,\,0<x\le 1,\, 0<y\le 1\}
$$
and consider  the sets $En(\Si)$ and $Ex(\Si)$ defined  in \eqref{ens}.
We are going to denote the unstable branches of $p_i$ by
$W^u_+(p_i)$ and $W^u_-(p_i)$ and the stable ones by $W^s_+(p_i)$
and $W^s_-(p_i)$.

All branches of all points $p_i$ accumulate on $\Si$.

Let $q^u_+(p_i)\in W^u_+(p_i)$ and $q^u_-(p_i)\in W^u_-(p_i)$ be the first point
of each unstable branch to intersect $\Si$ as we move along the corresponding
branch starting from $p_i$. Both $q^u_+(p_i)$ and $q^u_-(p_i)$ belong to
$En(\Si)$. Join $q^u_+(p_i)$ and $q^u_-(p_i)$ by a small arc $\ga^u_i$ contained 
in $En(\Si)$. Let $W^u[q^u_-(p_i),q^u_+(p_i)]$ be the segment inside
$W^u(p_i)$ from $q^u_-(p_i)$ to $q^u_+(p_i)$.
Then we have that $\Ga^u_i=W^u[q^u_-(p_i),q^u_+(p_i)]\cup\ga^u_i$ is a 
simple closed curve that contains $p_i$.

In the same way let $q^s_+(p_i)\in W^s_+(p_i)$ and $q^s_-(p_i)\in W^s_-(p_i)$
be the first point of each stable branch to intersect $\Sigma$ as we move along 
the corresponding branch starting from $p_i$.
 Both $q^s_+(p_i)$ and $q^s_-(p_i)$ belong to $Ex(\Si)$.
Join $q^s_+(p_i)$ and $q^s_-(p_i)$ by a small arc $\ga^s_i$ contained in $En(\Si)$.
Let $W^s[q^s_-(p_i),q^s_+(p_i)]$ be the segment inside $W^s(p_i)$ from 
$q^s_-(p_i)$ to $q^s_+(p_i)$. Then we have that 
$\Ga^s_i=W^s[q^s_-(p_i),q^s_+(p_i)]\cup \ga^s_i$
is a simple closed curve that contains $p_i$.

We are going to assume that $p$ has no homoclinic points and derive a contradiction.

We have that $\ga^u_i\cap \ga^s_j=\emptyset$ for all pairs $(i,j)$.
Therefore $\Ga^u_i\cap \Ga^s_j=\emptyset$ if $i\ne j$ and 
$\Ga^u_i\cap \Ga^s_i=\{p_i\}$.

If $\a$ and $\be$ are oriented closed curves in $\T^2$ let 
$\#(\a,\be)$ be the oriented intersection number of $\a$ and $\be$.
Then $\#$ depends only on the homology classes of $\a$ and $\be$ and is a non degenerate
skew symmetric bilinear form over the integers defined on $H_1(\T^2)=\Z^2$.
We have that $E\subset \T^2$ and if we look at $\Ga^u_i$ and $\Ga^s_i$ as elements of 
$H_1(\T^2)$, then
\begin{equation}\label{intGa}
|\#(\Ga^u_i,\Ga^s_j)|=
\begin{cases}
1 &\text{if } i=j,
\\
0&\text{if } i\ne j.
\end{cases}
\end{equation}
The rank\footnote{In other surfaces their cyclic covers do not have the same Betti numbers as 
in the base and then the following argument does not apply.}
 of $H_1(\T^2)$ is $2$ and therefore $\{\Ga^s_1,\Ga^s_2,\Ga^s_3\}$ is 
linearly dependent. So one of them is a linear combination of the other two,
say $\Ga^s_1=a\Ga^s_2+b\Ga^s_3$. It follows that
$$
|\#(\Ga^u_1,\Ga^s_1)|\le|a||\#(\Ga^u_1,\Ga^s_2)|+|b||\#(\Ga^u_1,\Ga^u_3)|=0,
$$
which contradicts~\eqref{intGa}.

This proves the existence of a homoclinic point  if alternative~\eqref{t631} holds.

Suppose now that there exist pairwise disjoint neighborhoods $V_i$ of $p_i$
such that $K_i\cap V_j=\emptyset$ if $i\ne j$. 
The procedure is almost the same.
For each $i$ we consider a sector $\Si_i$ of $p_i$.
The branches of one $p_i$ do not intersect $\Si_j$ if $i\ne j$.
The curves $\Ga^u_i$ and $\Ga^s_i$ are constructed by considering 
the first point of each branch of $p_i$ to intersect $\Si_i$.
They also satisfy~\eqref{intGa}, and we derive a contradiction in the same way.

In order ot complete the proof of item~\eqref{t63},
we need to show that once $p$ has a homoclinic point
then the four branches of $p$ all have homoclinics points.
We say that two curves $\ga_1$, $\ga_2$ cross if they 
intersect at a point $q\in\ga_1\cap\ga_2$ and there is a 
contractible neighborhood $U$ of $q$ such that $U-\ga_1$ is disconnected and 
$U\cap \ga_2$ intersects different connected components of $U-\ga_1$.

%\psfrag{Lu}{\Huge $L^u$}
%\psfrag{Ls}{\Huge $L^s$}
%\psfrag{p}{\Huge $p$}
%\psfrag{q}{\Huge $q$}

%\begin{figure}[h]
%\resizebox*{6cm}{4cm}
%{\includegraphics{grid2.pdf}} 
%\caption{\footnotesize This figure illustrates how a crossing of invariant manifolds produces a grid of
%invariant manifolds near the fixed point. 
%}\label{grid}
%\end{figure}

\begin{Lemma}\label{Lall}
If two branches of $p$  cross, then all branches of $p$ cross.
\end{Lemma}

\begin{proof}
Let $q$ be a homoclinic point of $p$ and let $\ga$ be the closed curve obtained 
by the union of the segments of the branches that contain $q$ from $q$ to $f(q)$.
If $\Ga=\ga\cup f(\ga)\cup f^2(\ga)$, then $S-\Ga$ is disconnected.
Let $A$ be the component of $S-\Ga$ that contains $p$ and its local branches 
and let $B$ be any other component.
Since $f$ is area preserving, there exists $n$ such that 
$f^nB\cap int_S B\ne \emptyset$.
Therefore $int_SB$ contains points of $fr_S(f^nB)$
which is made of invariant manifolds.
But the four branches have the same closure in $S$,
implying that each one intersects different components of $S-\Ga$.
It follows that each branch intersects $fr_S(S-\Ga)=\Ga$.
\end{proof}

%\begin{proof}
%Let $L_s\subset W^s(p)$ and $L_u\subset W^u(p)$ be branches of $p$ which 
%cross at $q\in L_s\cap L_u$. By hypothesis $L_s\cup L_u\subset S_0$.
%Iterating small segments in $L_s$ and $L_u$ which contain $q$ and which
%cross $L_u$ and $L_s$ respectively we obtain a grid in a neighborhood
%of $p$. Let $Q_1$, $Q_2$  be small topological rectangles nearby $p$
%whose boundaries are in $L_s\cup L_u$ and such that 
%$cl_SQ_1\subset int_SQ_2$ and $p\notin cl_S Q_2$.
%
%Let $L$ be another branch of $p$. By hypothesis $cl_SL=cl_SL_s=cl_SL_u$.
%Therefore $L$ accumulates on $fr_S Q_1$ and hence $L\cap int_S Q_2=\emptyset$.
%This implies that $L\cap fr_S Q_2\ne \emptyset$ and hence $L$ has a homoclinic
%point, which is also a crossing.
%
%\end{proof}

\subsection{Proof of item~\eqref{t83}: the case of  boundary fixed points. }
\label{ssp84}
\quad

\noindent{\it Case 1:  $S$ has genus 0.}\quad

Suppose now that $S$ has genus zero.

Let $p_1$ and $p_2$ be any fixed points of $f$ in $C$, let $L_1$ and $L_2$ be their branches 
contained in $S-\partial S$ and assume that $L_1$ is unstable and $L_2$ is stable.
Since $cl_SL_1=cl_SL_2$ and $L_2$ accumulates on both sectors of $p_2$ adjacent to itself,
we have that $L_1$ accumulates on both sectors of $p_2$ adjacent to $L_2$.

Consider a system of coordinates in a neighborhood
$V$ of $p_2$ with $p_2$ at the origin  in which $f(x,y)=(\la^{-1}x,\la y)$
with $0<\la<1$, $y\ge 0$. 
We have that $L_2$ is the branch that contains 
$\{(x,y)\in V : x=0,\,y>0\}$. Consider the sector 
$$
\Si=\{(x,y)\in V : 0<xy\le\la^2,\,
0< x\le 1, \, 0<y\le1\}
$$ 
of $p_2$ and the sets $En(\Si)$ and $Ex(\Si)$ as defined in~\eqref{ens}. The branch
$L_1$ intersects $\Si$ and if we move along $L_1$ starting from $p_1$
the first point $q_1$ of $L_1$ to intersect $\Si$ belongs to $En(\Si)$.
Join $q_1$ to $p_2$ with a line segment $\ga_1\subset int_S\Si$ and
consider $\Ga_1=L_1[p_1,q_1]\cup \ga_1\cup C_{21}$, where $C_{21}$ is the 
component of $C-\{p_1,p_2\}$ that contains the branch of $p_2$
that contains $\{(x,y)\in V : x>0,\,y=0\}$. $\Ga_1$ is a simple closed curve.

 The branch $L_2$ intersects $\Si$ and as we move along $L_2$ starting from $p_2$,
 the first point $q_2$ of $L_2$ to intersect $\Si$ belongs to $Ex(\Si)$.
 Join $q_2$ to $p_2$ with a line segment $\ga_2\subset int_S\Si$ and consider
 the simple closed curve $\Ga_2=L_2[p_2,q_2]\cup\ga_2$.
 $\Ga_1$ and $\Ga_2$ intersect at $p_2$. The local branch of $L_2$ and the curve $\ga_2$
 belong to different components of $S-\Ga_1$.
 From this we conclude that $\Ga_1$ and $\Ga_2$ 
 must intersect again at a point $q$ different from $p_2$.
 Since $\Ga_1\cap \ga_2=\{p_2\}$ and $\Ga_2\cap(\ga_1\cup C_{21})=\{p_2\}$,
 we have that $q\in L_1[p_1,q_1]\cap L_2[p_2,q_2]$.
 This proves the case of genus zero.
 
\noindent{\it Case 2: $S$ has genus 1 and $C$ has at least 4 fixed points:}\quad
 
 Suppose now that $S$ has genus one.
 
 Let $p_1$, $p_2$, $p_3$, $p_4$ be fixed points of $f$ such that $p_i$
 and $p_{i+1}$ are end points of connections contained in $C$ for $1\le i\le 3$.
 Let $L_i$ be the branch of each $p_i$ contained in $S-\partial S$ and assume
 that $L_1$ and $L_3$ are stable and $L_2$ and $L_4$ are unstable.
By hypothesis $cl_SL_i=cl_SL_j$ for any $i,j$.
 Let $\pi:E\to S$ be a 3-fold covering of $S$ defined after~\eqref{3fold}
 and let $F:E\to E$ be a lifting of $f$.

Let $\pi^{-1}(C)= C^1\cup C^2\cup C^3$ be the lifting of 
the boundary component; $\pi^{-1}(p_i)=\{p^1_i,p^2_i,p^3_i\}$,
$1\le i\le 4$, the lifting of the fixed points $p_i$ and
$\pi^{-1}(L_i)=\{L^1_i,L^2_i,L^3_i\}$, $1\le i\le 4$, the lifting of the
corresponding branches.

To show that any pair $(L_i,L_j)$ of stable and unstable branches intersect,
it is enough to show that this happens for one pair.
This follows from the fact that they all have the same closure in $S$ 
and an argument similar to Lemma~\ref{Lall}. 
To show that a pair $(L_i,L_j)$ of stable and unstable branches intersect,
it is enough to show that this happens for a pair
$(L^m_i,L^n_j)$ of stable and unstable branches in $E$.

We are going to assume that this does not happen and derive a contradiction.

We would like to construct simple closed curves such as $\Ga^u_i$ and $\Ga^s_j$
in the proof of item~\eqref{t82} in \S~\ref{ssp82}, with the same intersecting properties.
The same proof as in Lemma~\ref{l47} shows that for each $n\in\{1,2,3\}$  the branches 
$L_1^n$, $L_2^n$, $L_3^n$, $L_4^n$ have the same closure $K^n$ in $E$
and that there are two alternatives:

\begin{enumerate}
\iitem\label{t641}
There exist a pair $(m,n)$ with $m\ne n$ and a pair $(i,j)$ such that a branch
$L^m_i$ accumulates on a sector of $p^n_j$. In this case we have that all branches
$L^m_i$ accumulate on all sectors of the points $p^n_j$ and $K^1=K^2=K^3$.
This follows from the action of deck transformations on fibers.
\iitem\label{t642}
For every $i$,~$j$ and $m\ne n$, the branches $L^m_i$ do not accumulate on
any of the sectors of points $p^n_j$. In this case there exist neighborhoods
$V^n_i$ of $p^n_i$ such that $K^m\cap V^n_i=\emptyset$ if $m\ne n$.
\end{enumerate}

Suppose that alternative~\eqref{t641} holds.

For $n=1,2,3$ let $C^n_{14}$ be the closed arc of $C^n$ that has $p^n_1$ and $p^n_4$
as end points and contains $p^n_2,\, p^n_3$ and the three connections between them.
For $1\le i<j\le 4$, let $C^n_{ij}$ be the closed arc contained in $C^n_{14}$ that has $p^n_i$
and $p^n_j$ as end points.

Let $\Si$ be the sector of $p^1_1$ that has $L^1_1$ and the connection contained
in $C_{12}^1$ as adjacent branches:
$$
\Si=\{(x,y)\in V : 0<xy\le \la^2,\, 
0<x\le 1,\, 0<y\le 1\}
$$
in local linearizing coordinates.

All branches $L^n_i$ intersect $\Si$. Let $q^n_i$ be the first point of $L^n_i$
to intersect $\Si$. If $i\in\{1,3\}$ then $L^n_i$ is stable and $q^n_i\in Ex(\Si)$. 
If $i\in\{2,4\}$ then $L^n_i$ is unstable and $q^n_i\in En(\Si)$.
Let $L^n_i(p^n_i,q^n_i)$ be the open arc inside $L^n_i$ between 
$p^n_i$ and $q^n_i$.

If we define 
$\Ga^n_{13}$ as $L^n_1(p^n_1,q^n_1)\cup C^n_{13}\cup L^n_3(p^n_3,q^n_3)\cup
\ga^n_{13}$ where $\ga^n_{13}$ is a small arc from $q^n_1$ to $q^n_3$
contained in $Ex(\Si)$, then $\Ga^n_{13}$ is a simple closed curve made of
parts of stable branches, an arc of $C^n$ and a small arc contained in $Ex(\Si)$.

Similarly, let 
$\Ga^n_{24}=L^n_2(p^n_2,q^n_2)\cup C^n_{24}\cup L^n_4(p^n_4,q^n_4)\cup \ga^n_{24}$,
where $\ga^n_{24}$ is an arc from $q^n_2$ to $q^n_4$ contained in $En(\Si)$.
Then $\Ga^n_{24}$ is a simple closed curve made of parts of unstable branches,
an arc of $C^n$ and a small arc contained in $En(\Si)$.

From the assumption that different branches $L^n_i$ and $L^m_j$ do not intersect,
we have that $\Ga^n_{13}\cap \Ga^m_{24}=\emptyset$ if $n\ne m$ and 
$\Ga^n_{13}\cap \Ga^n_{24}=C^n_{23}$ for $n=1,2,3$.

We have that $E\subset \T^2$ and if we look at the curves $\Ga^n_{ij}$ as subsets of $\T^2$,
then $\{\Ga^1_{13},\Ga^2_{13},\Ga^3_{13}\}$ and $\{\Ga^1_{24},\Ga^2_{24},\Ga^3_{24}\}$
are linearly dependent subsets of $H_1(\T^2)$. We have that  $\#(\Ga^n_{13},\Ga^m_{24})=0$
if $n\ne m$.

We claim that $|\#(\Ga^n_{13},\Ga^n_{24})|=1$ for $n=1,2,3$.

In order to see why this happens we are going to homotopicaly modify
each $\Ga^n_{24}$ so that $\Ga^n_{13}$ and $\Ga^n_{24}$ intersect 
only at $p^n_3$ and this point is a crossing, meaning that there exists a neighborhood
$V$ of $p^n_3$ such that the complement in $V$ of each curve has two
components, with the other curve intersecting both.

Each boundary component $C^n$ has a collar neighborhood $W$ homeomorphic
to $(\re/5\Z)\times [0,1[$ where we take coordinates $(\th,r)$. 
Points in $C^n$ have coordinate $r=0$. The following refers to the coordinates $(\th,r)$.

We may assume that $p^n_i=(i,0)$. Therefore $\Ga^n_{13}\cap \Ga^n_{24}=[2,3]\times\{0\}$.

For any $\de_1>0$ there exists $\de_2>0$ such that 
$R=(]2-\de_1,3-\de_1[\times]0,\de_2[)$ is disjoint from $\Ga^n_{13}$ 
and contains the local branch of $L^n_2$. Let $q_0\in R$ be a point of
the local branch of $L^n_2$ inside $R$. It follows that the straight line
closed segment $\a_1$ from $q_0$ to $q_1:=(3-\de_1,0)$ is contained 
in $R\cup\{q_1\}$. Let $\a_0$ be the closed segment inside $\Ga^n_{24}$ from $q_0$
to $q_1$ that contains $p^n_2$. The arc $\a_0$ is the union of
a segment of the local branch of $L^n_2$ with $[2,3-\de_1]\times\{0\}$.
Let $\a:[0,1]^2\to S$ be the straight line path homotopy 
$\a(s,t)=(1-t)\,\a_0(s)+t\,\a_1(s)$ from $\a_0$ to $\a_1$.

Now let us consider a neighborhood $V$ of $p^n_3$ 
and coordinates in $V$ with $p^n_3=(0,0)$ and $f(x,y)=(\la^{-1}x,\la y)$, $y\ge 0$
and $\la <1$. We may assume that $V$ is an open ball
with center $(0,0)$ in these coordinates. If $\de_1$ is small enough,
we have that $q_1\in V$. Therefore for $a<1$ and close enough to $1$, we have that
$\be_a(t):=\a(a,t)\in V$ for all $t\in[0,1]$. Let $\be_1(t)=p^n_3$ and let
$\be:[a,1]\times[0,1]\to S$ be the straight line path homotopy from
$\be_a$ to $\be_1$ given by
$$
\be(s,t)=\tfrac{1-s}{1-a}\,\be_a(t)+\tfrac{s-a}{1-a}\,\be_1(t).
$$

We have that $\a(a,t)=\be(a,t)$ and therefore $\a$ and $\be$ 
define a continuous map
\linebreak
$\ga:[0,1]^2\to S$ by $\ga(s,t)=\a(s,t)$ if $s\le a$
and $\ga(s,t)=\be(s,t)$ if $s\ge a$.

Then $\ga$ is a path homotopy between $\ga_0(s)=\ga(s,0)$ and 
$\ga_1(s)=\ga(s,1)$, where
$$
\ga_0(s) =
\begin{cases}
\a_0(s) &\text{if } s\le a,
\\
\tfrac{1-s}{1-a}\,\a_0(a) + \tfrac{s-a}{1-a}\,p^n_3
&\text{if } s\ge a
\end{cases}
$$
is the arc of $\Ga^n_{24}$ from $q_0$ to $p^n_3$ that contains $p^n_2$, and
$$
\ga_1(s)=
\begin{cases}
\a_1(s) &\text{if } s\le a,
\\
\tfrac{1-s}{1-a}\,\a_1(a) + \tfrac{s-a}{1-a}\,p^n_3
&\text{if } s\ge a.
\end{cases}
$$

We have that $\ga_1(0)=q_0$, $\ga_1(1)=p^n_3$ and since $\a_1$ is a straight line
in the coordinates of $R$ we have that $\ga_1([0,a])\subset R\subset S-\Ga^n_{13}$.
We also have that $\ga_1(1)=p^n_3$ and that $\ga_1(a)$ belongs to the sector 
$\Si=\{(x,y)\in V : x<0,\,y>0\}$.
Since $\ga_1$ is a straight line in the coordinates of $V$ we have that 
$\ga_1([a,1[)\subset \Si\subset S-\Ga^n_{13}$.

Therefore $\ga_1([0,1])\cap \Ga^n_{13}=\{p^n_3\}$. 
Let us replace the arc of $\Ga^n_{24}$ from $q_0$ to $p^n_3$
that contains $p^n_2$ by $\ga_1([0,1])$ to obtain a new simple closed curve $\La$.

Then $\La$ and $\Ga^n_{24}$ are homotopic and $\La\cap \Ga^n_{13}=\{p^n_3\}$.

In the coordinates given by $V$ we have that 
$$
\Ga^n_{13}\cap V=\{(x,y)\in V : x\le 0,\, y=0\}\cup\{(x,y)\in V : x=0,\, y\ge 0\}
$$
and the complement of $\Ga^n_{13}$ in $V$ has two components
$\Si =\{(x,y)\in V : x<0,\, y>0\}$ and 
$\Si'=\{(x,y)\in V : x>0,\, y\ge 0\}$.

On the other hand $\La\cap V$ is the disjoint union of $\ga_1([0,1[)\cap\Si$,
$\{p^n_3\}$ and $\{(x,y)\in V : x>0,\, y=0\}$.
Therefore $\La$ intersects both $\Si$ and $\Si'$.

From this we conclude that $\#(\La,\Ga^n_{13})$ is $1$ or $-1$ which proves the claim.

Now the contradiction comes easily from the fact that 
$\Sigma_{n=1}^3 a_n\Ga^n_{13}=0$, with $a_m\ne 0$ for some $m$;
implying that 
$$
0=\#\big(\tsum\nolimits_{n=1}^3 a_n\Ga^n_{13},\Ga^m_{24}\big)=a_m\,\#(\Ga^m_{13},\Ga^m_{24}).
$$
A contradiction.

This proves item~\eqref{t83} 
in the case $S$ has genus one and alternative \eqref{t641}
holds.

When alternative~\eqref{t642} holds, we consider a sector $\Si_n$ of $p^n_1$ contained in $V^n_1$
for $1\le n\le 3$. For each $n$ the four branches $L^n_i$ intersect $\Si_n$ and are
disjoint from $\Si_m$ if $n\ne m$. Let $q^n_i$ be the first point of $L^n_i$ to intersect $\Si_n$.

For each $n=1,2,3$ we define a pair of simple closed curves 
$\{\Ga^n_{13},\Ga^n_{24}\}$ as in alternative~\eqref{t641}. 
The only difference is that the arcs $\ga^n_{13}$ and $\ga^n_{24}$ will be
contained in different sectors $\Si_n$. If there are no homoclinic points then 
using \eqref{t641} we 
still have that $\Ga^n_{13}\cap \Ga^m_{24}=\emptyset$ if $n\ne m$ and
$\Ga^n_{13}\cap \Ga^n_{24}=C^n_{23}$ for $n=1,2,3$,
and a contradiction comes up when we compute their oriented intersection number.
This ends the proof of Theorem~\ref{T8}.

\qed

The following corollary extends Theorem~\ref{T6} to periodic points.
It follows immediately by taking powers of $f$ and adding the hypothesis 
of elliptic periodic being irrationally elliptic.

\begin{Corollary}\label{C1}
\quad

Let $S$ be a compact connected orientable surface with boundary 
provided with a finite measure $\mu$ 
which is positive on open sets and let
$f:S\to S$ be an orientation preserving and  area preserving
homeomorphism of $S$.
\begin{enumerate}
\iitem\label{c11}
Suppose that $L$ is a (periodic) 
branch of $f$ and that all periodic points of $f$ contained in $cl_SL$ 
are of saddle type or irrationally elliptic.
Then either $L$ is a connection or $L$ accumulates on
both adjacent sectors.
In the later alternative $L\subset \om(L)$.

\iitem\label{c12}
Let $p\in S-\partial S$ be a periodic point of $f$ of saddle type and let
$L_1$ and $L_2$ be adjacent branches of $p$ that are not connections.
If all the periodic points of $f$ contained in $cl_S(L_1\cup L_2)$ are of saddle
type or irrationally elliptic, then $cl_SL_1=cl_SL_2$.

\iitem\label{c13}
Suppose that $p\in S-\partial S$ is a periodic point of $f$ of saddle type.
Assume that all the periodic points contained in $cl_S(W^u_p\cup W^s_p)$ 
are of saddle type or irrationally elliptic and $p$ has no connections.
Then the branches of $p$ have the same closure
and each branch of $p$ accumulates on all the sectors of $p$.

If in addition $S$ has genus zero or one, then the four 
branches of $p$ have homoclinic points.

\iitem\label{c14}
Let $C$ be a connected component of $\partial S$ and suppose that   
all the periodic points $p_1,\ldots,p_{2n}$ of $f$ in $C$ are of saddle type.
Let $L_i$ be the branch of $p_i$ contained  in $S-\partial S$.
Assume that for every $i$ all the periodic points of $f$ contained 
in $cl_SL_i$ are of saddle type or irrationally elliptic and that $L_i$ is not a connection.
Then for every pair $(i,j)$  the branch $L_i$ accumulates on all the sectors  $p_j$ and
 $cl_SL_i=cl_SL_j$.

If in addition $S$ has genus zero  then any pair $(L_i,L_j)$
of stable and unstable branches intersect.
The same happens if the genus of $S$
is 1 provided that there are at least $4$ periodic points in $C$.

\end{enumerate}

\end{Corollary}

The version of the Corollary for maps defined on an open subset is the following:

\begin{Corollary}\quad

Let $S$ be a compact connected orientable surface with boundary. 
Let $S_0\subset S$ be a submanifold with compact boundary 
$\partial S_0\subset \partial S$
 and let $f,f^{-1}:S_0\to S$
be an orientation preserving and area preserving homeomorphism of $S_0$ onto
open subsets $fS_0$, $f^{-1}S_0$ of $S$ with $f(\partial S_0)\subset \partial S_0$.

\begin{enumerate}
\iitem 
Let $p\in S_0-\partial S$ be a periodic point of $f$ of saddle type.
Assume that the branches of $p$ have closure
included in $S_0$.
Assume also that each branch of $p$ accumulates on both
of its adjacent sectors and that all the branches of $p$
have the same closure in $S$.
If in addition $S$ has genus 0 or 1,
then the four branches of $p$ have homoclinic points.

\iitem
Let $C$ be a connected component of $\partial S_0$ and suppose that 
all the periodic points $p_1,\ldots,p_{2n}$ of $f$ in $C$
are of saddle type.
Let $L_i$ be the branch of $p_i$ contained in
$S-\partial S$.
Assume that for every $i$, $L_i$ is not a connection and
$cl_SL_i=cl_SL_j\subset S_0$ for every pair $(i,j)$.

If in addition $S$ has genus 0, then every pair $L_i$, $L_j$ of stable and 
unstable branches intersect. 
The same happens if the genus of $S$
is 1 provided that there are at least $4$ periodic points in $C$.
\end{enumerate}

\end{Corollary}

\section{The standard map.}
\label{s5}

The standard map is a one parameter family of area preserving
diffeomorphisms of the two dimensional torus $T^2=\re^2/\Z^2$ given by
$$
f_\la(x,y)=\big(x+y+\tfrac \la{2\pi}\sin(2\pi x),\;
y+\tfrac\la{2\pi}\sin(2\pi x)\big), \quad
\la\in \re.
$$
The map $f_0$ is just a twist. 
The map $\vr(x,y)=(x+\tfrac 12,y)$ 
is a conjugacy between $f_\la$ and $f_{-\la}$,
so we consider only parameters $\la>0$.

For $\la\ne 0$ there are two fixed points, $p=(0,0)$
and $q=(\tfrac 12,0)$. For $\la>0$,  $p$ is always a saddle
with positive eigenvalues, and  $q$ is elliptic if $0<\la< 4$ 
and a saddle with negative eigenvalues if $\la>4$.

\begin{Theorem}\label{T9}
If $\la \ne 4$, $\la>0$,  then the four branches of $p$ have homoclinic points.
\end{Theorem}

\begin{proof}

For $\la \ne 4$ if $L$ is a branch of $p$ then $L$ is invariant and all
fixed points of $f$ contained in $cl_SL$ are non-degenerate.
The identity $f(-x,-y)=-f(x,y)$ implies that the invariant manifolds 
of $p$ are symmetric with respect to $(0,0)$. 
Therefore if $L$ is a connection then so is $-L$.

There can not be a connection between $p$ and $q$. 
This is obvious when $q$ is elliptic, and when $q$ is a 
saddle this can not happen because the branches of $p$
are invariant and those of $q$ have period two.

Therefore if one of the branches of $p$ is a connection,
then this connection equals two branches of $p$.
By the symmetry shown above the invariant manifolds of $p$
are made of two connections of homoclinic points.
So in this case the four branches have homoclinic points. 

If no branch of $p$ is a connection, then by \eqref{t63} of Theorem~\ref{T6},
the four branches of $p$ have homoclinic points.
\newline
\end{proof}

\nocite{Massey}

\nocite{Ol3}

% \bibliographystyle{amsplain}
% \bibliography{biblio}

\def\cprime{$'$} \def\cprime{$'$} \def\cprime{$'$} \def\cprime{$'$}
\providecommand{\bysame}{\leavevmode\hbox to3em{\hrulefill}\thinspace}
\providecommand{\MR}{\relax\ifhmode\unskip\space\fi MR }
\providecommand{\MRhref}[2]{%
  \href{http://www.ams.org/mathscinet-getitem?mr=#1}{#2}
}
\providecommand{\href}[2]{#2}

  \end{document}